\documentclass[11pt]{amsart}
\usepackage[margin=3.39cm]{geometry}
\usepackage[numbered]{bookmark}
\usepackage{changepage} 
\usepackage[utf8]{inputenc}
\usepackage{graphics, graphicx}
\usepackage{amsmath, amsfonts, amsthm, amssymb}
\usepackage{mathrsfs, mathtools, mathabx}
\usepackage{verbatim} 
\usepackage{enumerate}
\usepackage{xcolor}
\usepackage{cite}

\usepackage{tikz}
\usetikzlibrary{calc}	
\usetikzlibrary{arrows, arrows.meta}

\newtheorem{theorem}{Theorem}
\newtheorem*{theorem*}{Theorem}
\newtheorem{definition}[subsection]{Definition}
\newtheorem{lemma}[subsection]{Lemma}
\newtheorem{proposition}[subsection]{Proposition}
\newtheorem{corollary}[subsection]{Corollary}

\newcommand{\Z}{\mathbb{Z}} 					
\newcommand{\N}{\mathbb{N}}					
\newcommand{\R}{\mathbb{R}}					
\newcommand{\C}{\mathbb{C}}					
\newcommand{\T}{\mathbb{T}}					

\newcommand{\A}{\mathcal{A}}									
\newcommand{\PermSpace}{\mathcal{G}_\A}						
\newcommand{\Simplex}{\Delta_\A}								
\newcommand{\IETSpace}{\PermSpace \times \R_+^\A}				
\newcommand{\IETSpaceNorm}{\PermSpace \times \Simplex}			
\newcommand{\Dom}{X_\A}									
\newcommand{\DomNorm}{\widetilde{X}_\A}						
\newcommand{\DomExt}{\mathfrak{X}_\A}							
\newcommand{\DomExtNorm}{\widetilde{\mathfrak{X}}_\A}			
								
\newcommand{\RV}{\mathcal{RV}}								

\newcommand{\ZorichMap}{\mathcal{Z}}							
\newcommand{\ZorichNorm}{\widetilde{\ZorichMap}}					
\newcommand{\ZorichMapExt}{\ZorichMap_{\textup{ext}}}			
\newcommand{\ZorichNormExt}{\ZorichNorm_{\textup{ext}}}
\newcommand{\IET}{(\pi, \lambda)}
\newcommand{\IETExt}{(\pi, \lambda, \tau)}
\newcommand{\IETExtn}[1]{(\pi^{#1}, \lambda^{#1}, \tau^{#1})}

\newcommand{\SBS}[2]{\mathcal{S}_{#1}#2}						

\newcommand{\Saff}{\textup{Aff}(\gamma, \omega)}

\newcommand{\Function}[5]{\begin{array}{cccc} #1 : & #2 & \rightarrow & #3 \\ & #4 & \mapsto & #5 \end{array}}

\newcommand{\orth}[1]{#1^\bot}
\newcommand{\ZC}[2]{{}A^T_{#1, #2}}
\newcommand{\HC}[2]{{}B^T_{#1, #2}}

\newcommand{\LC}[2]{B^{-1}_{#1, #2}}

\newcommand{\HRC}[2]{{}A^T_{#1, #2}}
\newcommand{\LRC}[2]{A^{-1}_{#1, #2}}

\newcommand{\Kpi}{\textup{Ker}(\Omega_\pi)}
\newcommand{\Kpim}[1]{\textup{Ker}(\Omega_{\pi^{#1}})}

\tikzstyle{vector}=[->,very thick,blue!70!black,line cap=round]
\tikzstyle{cone}=[thin,blue!50!black,fill opacity=0.8]

\newcommand\jetcone[4]{
    \pgfmathanglebetweenpoints{\pgfpointanchor{#1}{center}}{\pgfpointanchor{#2}{center}}
    \edef\tmpang{\pgfmathresult}
    \coordinate (tmpO) at ($(#1)+(\tmpang:0.02)$); 
    \coordinate (tmpC) at ($(#2)+(\tmpang-180:{abs(#4)+0.4})$); 
    \coordinate (tmpL) at ($(tmpC)+(\tmpang+90:#3)$); 
    \coordinate (tmpLL) at ($(tmpC)+(\tmpang+90.3:#3)$); 
    \coordinate (tmpR) at ($(tmpC)+(\tmpang-90:#3)$); 
    \coordinate (tmpRR) at ($(tmpC)+(\tmpang-90:#3)$); 
    \fill[blue!50!black, opacity=0.7, top color=blue!50!black!10,bottom color=blue!40!black!20,shading angle=10\tmpang,rotate=\tmpang]   (tmpLL)  arc(94:266.75:{#4} and {#3+0.002})  --  (tmpO) -- cycle;
    \fill[blue!50!black,fill, opacity=0.35, top color=blue!50!black!10,bottom color=blue!50!black!0,shading angle=\tmpang,rotate=\tmpang]  (tmpC) ellipse({#4} and {#3});
    \begin{scope}
    	\clip[rotate=\tmpang] (tmpRR) -- (tmpO) -- (tmpL) arc(90:-90:{#4} and {#3});
    \end{scope}
    \draw[thin,blue!50!black, rotate=\tmpang]    (tmpLL)  arc(94:-94.75:{#4} and {#3+0.002})  --  (tmpO) -- cycle;
}

\newcommand\jetconee[4]{
    \pgfmathanglebetweenpoints{\pgfpointanchor{#1}{center}}{\pgfpointanchor{#2}{center}}
    \edef\tmpang{\pgfmathresult}
    \coordinate (tmpO) at ($(#1)+(\tmpang:0.0)$); 
    \coordinate (tmpC) at ($(#2)+(\tmpang-180:{abs(#4)+0.4})$); 
    \coordinate (tmpL) at ($(tmpC)+(\tmpang+90:#3)$); 
    \coordinate (tmpLL) at ($(tmpC)+(\tmpang+90.3:#3)$); 
    \coordinate (tmpR) at ($(tmpC)+(\tmpang-90:#3)$); 
    \coordinate (tmpRR) at ($(tmpC)+(\tmpang-90:#3)$); 
    \draw[thin,blue!50!black,fill opacity=0.2, top color=blue!50!black!50,bottom color=blue!50!black!60,shading angle=\tmpang,rotate=\tmpang]  (tmpC) ellipse({#4} and {#3});
    \begin{scope}
        \clip[rotate=\tmpang] (tmpRR) -- (tmpO) -- (tmpL) arc(90:-90:{#4} and {#3});
    \end{scope}
    \draw[thin,blue!50!black,fill opacity=0.8, top color=blue!50!black!30,bottom color=blue!40!black!50,shading angle=\tmpang,rotate=\tmpang]   (tmpLL)  arc(94:266.75:{#4} and {#3+0.002})  --  (tmpO) -- cycle;
}

\date{}
\begin{document}
\sloppy
\title[Affine IETs with a singular conjugacy to an IET]{Affine interval exchange maps with a singular conjugacy to an IET}
\author{Frank Trujillo and Corinna Ulcigrai}
\maketitle

\begin{abstract}
We produce affine interval exchange transformations (AIETs) which are \emph{topologically} conjugated to (standard) interval exchange maps (IETs) via a \emph{singular conjugacy}, i.e.~a diffeomorphism $h$ of $[0,1]$ which is $\mathcal{C}^{0}$ but not $\mathcal{C}^{1}$ and such that the pull-back of the Lebesgue measure is a \emph{singular} invariant measure for the AIET. 
In particular, we show that for almost every IET $T_0$ of $d\geq 2$ intervals and any vector $\omega$ belonging to the central-stable space $E_{cs}(T_0)$, for the Rauzy-Veech renormalization, any AIET $T$ with log-slopes given by $\omega$ and semi-conjugated to $T_0$ is topologically conjugated to $T$. In addition, if $\omega \notin E_s(T_0)$, the conjugacy between $T$ and $T_0$ is singular.
\end{abstract}

\section{Introduction and main results}
The study of circle diffeomorphisms is a classical topic in dynamical systems, initiated by H. Poincaré (we refer, for example, to \cite{katok_introduction_1995}, \cite{sinai_topics_1994} or \cite{de_melo_one-dimensional_1993} for a basic overview). 

Two fundamental questions addressed by the theory of circle diffeomorphisms are the \emph{existence} and the \emph{regularity} of a topological conjugacy between a minimal circle diffeomorphism $f$ and its linear, {isometric model (which is a rigid rotation $R_\alpha$, where $\alpha $ is the rotation number $\alpha$ of $f$)}, namely of a homeomorphism $h$ such that $h\circ f= R_\alpha \circ h$  (see for example M. Herman's work \cite{herman_sur_1979}). 

The initial motivation for Poincaré to study circle diffeomorphisms is that they appear naturally as first-return maps of flows on surfaces of genus one. \emph{Interval exchange transformations}, or IETs (more precisely \emph{generalized} IETs as well as, as special cases, affine or standard IETs) for short, appear as first-return maps of flows on surfaces and are thus seen as natural generalizations of circle diffeomorphisms to higher genus (with rigid rotations and affine circle diffeomorphisms, in turn, generalizing IETs and affine IETs, respectively). Thus, it is natural to ask to what extent the theory of circle diffeomorphisms extends to generalized interval exchange maps. 
 Efforts in this direction have been ongoing since the early eighties and this is currently an active area of research, see for example \cite{levitt_decomposition_1987, forni_solutions_1997, marmi_cohomological_2005, marmi_affine_2010, marmi_linearization_2012, ghazouani_local_2020, ghazouani_priori_2021}. 
 We refer the reader to the articles \cite{marmi_linearization_2012, ghazouani_priori_2021} for further references and a more in-depth discussion about linearization and rigidity questions for GIETs. 

In this paper, we give a contribution to the study of the regularity of conjugacies between an affine interval exchange transformation (AIET) and its linear (piecewise) isometric model, namely, a (standard) interval exchange transformation (IET). In particular, we produce AIETs that are conjugated to a standard minimal interval exchange transformation via a conjugacy $h$ which is $\mathcal{C}^0$ but fails to be $\mathcal{C}^1$. These AIETs are uniquely ergodic, and the unique invariant measure is \emph{singular} with respect to the Lebesgue measure; in this case, we say that they have a \emph{singular conjugacy} to a (minimal) IET. 
A one-parameter family of examples of AIETs with a singular conjugacy to a minimal IET was constructed by Isabelle Liousse in \cite{liousse_echanges_2002}. We provide a criterium that allows constructing AIETs having singular conjugacy with its (piecewise) isometric model for full measure classes of IET rotation numbers (see the Theorem in \S~\ref{sec:regularity} below for an informal statement, as well as Theorems~\ref{thm:topconjugacy} and \ref{thm:regularity} in \S~\ref{sec:main} for precise results). 

AIETs (whose formal definition we postpone to \S~\ref{sec:AIETs}) can be seen as a generalization to higher genus of \emph{affine} -also known as \emph{piecewise linear} or for short PL- circle diffeomorphisms (defined in \S~\ref{sec:singular} below). In the setting of PL-circle diffeomorphisms, singularity of the conjugacy to the corresponding linear model is a well-known phenomenon, as the results summarized in \S~\ref{sec:singular} show. Contrary to (PL-)circle diffeomorphisms, though, for which the existence of a topological conjugacy follows from the classical work of A. Denjoy as long as there is sufficient regularity (see \S~\ref{sec:wandering}), for AIETs (and GIETs in general) semi-conjugated to a minimal IET, the existence of a topological conjugacy is not granted, i.e.~regularity assumptions are not sufficient to exclude the presence of wandering intervals: 
several results (see \S~\ref{sec:wandering}) 
show not only the existence but also the ubiquity of wandering intervals in AIETs. 
Therefore, a crucial part of the present paper 
is to prove a criterion for the absence of wandering intervals which can be applied to full measure sets of IETs rotation numbers. 

We now summarize some of the results in the literature concerning the singularity of conjugacies of affine circle diffeomorphisms (\S~\ref{sec:singular}), and the existence of wandering intervals in circle homeomorphisms and AIETs (\S~\ref{sec:wandering}). An informal statement of the main result of this paper is given at the end of this introduction, in  \S~\ref{sec:regularity}.



\subsection{Singular conjugacies in PL setting}\label{sec:singular}
It is well known that sufficiently regular circle diffeomorphisms are smoothly conjugated to the corresponding linear rotation for a full measure set of rotation numbers, a celebrated result proved by M.~Herman \cite{herman_sur_1979} and later extended by J.-C.~Yoccoz \cite{yoccoz_conjugaison_1984} to all Diophantine rotation numbers (thus providing the optimal arithmetic condition).

M.~Herman  \cite{herman_sur_1979} was the first to show that a circle homeomorphism with irrational rotation number that is piecewise linear and has exactly two points (called \emph{break points}) where the first derivative is discontinuous has an invariant measure absolutely continuous with respect to Lebesgue if and only if its break points lie on the same orbit.

More generally, a homeomorphism of the circle $f: \T \to \T$ is called a \emph{piecewise smooth circle homeomorphism} or a \emph{P-homeomorphism} if it is a smooth orientation preserving homeomorphism, differentiable away from countable many points, so-called \emph{break-points}, at which left and right derivatives, denoted by $Df_-$, $Df_+$ respectively, exist but do not coincide, and such that $\log Df$ has bounded variation. 
A P-homeomorphism which is \emph{linear} (i.e.~\emph{affine}) in each domain of differentiability is called a \emph{PL-homeomorphism}. 

In \cite{liousse_nombre_2005}, Isabelle Liousse showed that the invariant measure of a generic 
PL-homeomorphism with a finite number of break points and irrational rotation number of bounded type is singular with respect to Lebesgue. The generic condition in \cite{liousse_nombre_2005} is explicit and appears as an arithmetic condition on the logarithm of the slopes of the PL-homeomorphism. For general P-homeomorphisms with exactly one break point and irrational rotation number, A.~Dzhalilov and K.~Khanin \cite{dzhalilov_invariant_1998} showed that the associated invariant probability measure is singular with respect to Lebesgue. The case of two break points has been studied by A. Dzahlilov, I. Liousse \cite{dzhalilov_circle_2006} in the bounded rotation number case, and by A. Dzahlilov, I. Liousse and D. Mayer \cite{dzhalilov_singular_2009} for arbitrary irrational rotation numbers. In both works, the authors conclude the singularity of the associated invariant probability measure. 

More recently, for P-homeomorphisms of class $C^{2 + \epsilon}$ with a finite number of break points and nonzero mean nonlinearity, K. Khanin and S. Kocić \cite{khanin_hausdorff_2017} showed that the Hausdorff dimension of their unique invariant measure is equal to $0$, provided that their rotation number belongs to a certain (explicit) full-measure set of irrational numbers. In the same work, the authors show that this result cannot be extended to all irrational rotation numbers.

\subsection{Existence and absence of wandering intervals}\label{sec:wandering}
Given a piecewise continuous map $f: I \to I$ defined on a compact interval, a subinterval $J \subset I$ is said to be a \emph{wandering interval} of $f$ if the forward iterates of $J$ by $f$ are pairwise disjoint. It is also common in the literature to include in the definition of a wandering interval the request that the $\omega$-limit set of $J$ is not finite. Their existence or absence plays an important role in one-dimensional dynamics, which has been widely studied in different settings. 

A celebrated theorem of A. Denjoy \cite{denjoy_sur_1932} shows that sufficiently smooth circle diffeomorphisms with irrational rotation number (more precisely, as soon as the logarithm of the derivative has bounded variation) do not admit wandering intervals. J.-C.~Yoccoz \cite{yoccoz_il_1984} proved that Denjoy's result remains valid for sufficiently smooth circle homeomorphisms with non-flat critical points, in particular, for analytic circle homeomorphisms. In the context of smooth interval transformations, M.~Martens, W.~de Melo, and S.~van Strien \cite{martens_julia-fatou-sullivan_1992} obtained a more general version of the previous result by showing that any $C^2$ map on a compact interval with non-flat critical points has no wandering intervals. 

On the other hand, several examples of transformations with wandering intervals exist in the literature. In \cite{denjoy_sur_1932}, A.~Denjoy constructed examples of $C^1$ circle diffeomorphisms with irrational rotation number having wandering intervals. This result was later improved to $C^{2 - \epsilon}$ regularity by M.~Herman \cite{herman_sur_1979} and similar examples in the context of multimodal maps of the interval are well-known (see, e.g., \cite[\S 3]{de_melo_one-dimensional_1993}). G.~Hall \cite{hall_cinfty_1981} constructed an example of a $C^\infty$ circle map with at most two flat critical points admitting wandering intervals. Hall's construction was recently generalized by Liviana Palmisano \cite{palmisano_denjoy_2015}, who obtained similar examples for circle maps with a half-critical point. 

\smallskip
\noindent {\it Wandering intervals in AIETs.} G. Levitt \cite{levitt_decomposition_1987} showed the existence of non-uniquely ergodic affine interval exchange transformations having wandering intervals and raised the question of whether unique-ergodicity for this class of transformations would be enough to rule out the existence of wandering intervals. R. Camelier and C. Gutierrez gave a negative answer to this question in \cite{camelier_affine_1997}. The example by Camelier and Gutierrez was studied in detail by M. Cobo \cite{cobo_piece-wise_2002}. 
The same example was later generalized by X. Bressaud, P. Hubert, and A. Maass \cite{bressaud_persistence_2010}, and their techniques have been recently used by M. Cobo, R. Gutierriez-Romo and A. Maass to show the existence of wandering intervals in a well-known transformation, the cubic Arnoux-Yoccoz map. 
{A condition for the absence of wandering intervals in the setting of substitutions (which include, in particular, AIETs with periodic rotation number) with unit eigenvalues is proved by X. Bressaud, A. Bufetov, and P. Hubert in \cite{bressaud_deviation_2014}.}
All the previous results in the setting of AIETs concern only an exceptional class among these transformations (namely those whose combinatorial rotation number is periodic), but in \cite{marmi_affine_2010}, S.~Marmi, P.~Moussa, and J.-C.~Yoccoz approached the general case and showed that AIETs which are semi-conjugated to a minimal IET of $d\geq 4$ intervals, under a full measure condition on the IET, possess wandering intervals.


\subsection{Regularity and Oseledets flags}\label{sec:regularity}
Let $T$ be an AIET with $d\geq 2$ continuity intervals which we assume is semi-conjugated to a minimal IET $T_0$. The IET $T_0$ can be thought of as a \emph{combinatorial} (IET) \emph{rotation number} for $T$, namely, it encodes combinatorial information on the structure of orbits of $T$ (see \S~\ref{sec:RV} for a precise definition) and, assuming that $T_0$ is minimal, it plays the role of \emph{irrationality} of the (IET) rotation number. In addition to the IET rotation number, the AIET is determined by the vector of \emph{slopes} $s=(s_i)_{i=1}^d \in \R^d_{+}$, recording the slope $s_i$ of each affine branch of $T$ (see \S~\ref{sec:AIETs}). Let $w =(w_i)_{i=1}^d\in \mathbb{R}^d$ denote the \emph{log-slope vector} of $T$, whose entries are given by $w_i:=\log s_i$ for $1\leq i\leq d$. 

A key realization by M. Cobo in \cite{cobo_piece-wise_2002} is that to study wandering intervals as well as the regularity of conjugacies for AIETs (under a full measure condition on the combinatorial rotation number), it is essential to know the position of the log-slope vector $\omega$ in the Oseledet's filtration of the \emph{Zorich cocycle} (a celebrated tool in the study of IETs which provide a multi-dimensional generalization of the continued fraction entries, see \S~\ref{sec:RV}). { The action of the Zorich cocycle on $\omega$ describes indeed how log-slopes change under renormalization (see \S~\ref{sec:RV}). One can show that for a.e.~IET $T_0$ on $d \geq 2$ intervals, there exist subspaces $\{0\}\subsetneq E_{s}(T_0) \subset E_{cs} (T_0) \subsetneq  \R^d$ (where $s$ and ${cs}$ stand for \emph{stable} and \emph{central stable}, respectively), such that if $\omega$ belongs to $E_s (T_0)$ (resp.~$E_{cs}(T_0)$) then the norm of the log-slopes decreases exponentially (resp. grows subexponentially) under renormalization, while if $ \omega\in \R^d \backslash E_{cs}(T_0) $, the norm of log-slopes grows exponentially. 

More precisely, the combination of several classical works  \cite{veech_gauss_1982, zorich_finite_1996, forni_deviation_2002, avila_simplicity_2007} shows that  the Zorich cocycle has $2g$ non-zero  Lyapunov exponents (where  $1 \leq g \leq \frac{d}{2}$ is determined by the combinatorics of the IET, see \eqref{eq:dimension_oseledets} in \S~\ref{sec:filtration}) of the form
$$\theta_{1}> \theta_2  >  \dots > \theta_g>0 > -\theta_g > \dots -\theta_2>-\theta_1,$$
and the Oseledets filtration of a generic $T_0$ has the form 
\[\R^d = E_{g} \supsetneq E_{g-1} \dots \supsetneq E_{1} \supsetneq E_{0} \supseteq E_{-1} \supsetneq  \dots \supsetneq  E_{-g+1}  \supsetneq E_{-g} \supsetneq \{0\},\]
where $E_i:= E_i(T_0)$ (resp.~$E_{-i}:= E_{-i}(T_0)$) is associated to the Lyapunov exponent $\theta_{g-i+1}$ (resp.~$-\theta_{-(g-i+1)}$), for $1\leq i\leq g$, and vectors in $E_0\backslash E_{-1}$ are associated to a zero Lyapunov exponent. We can then see that $E_{s}(T_0):=E_{-1}(T_0)$ and $E_{cs}(T_0):= E_0(T_0)$. 
The space $E_{-g}$ is called the \emph{strong-stable} space and we denote it by $E_{ss}= E_{ss}(T_0)$. 
 
We remark that $E_{cs}=E_s= E_{-1}$ (i.e.~there are no non-zero vectors associated with a zero exponent) if and only if $g = \frac{d}{2}$. }

\smallskip
 Under full measure conditions on $T_0$, the log-slope vector $\omega$ of $T$, which necessarily belongs to {$E_{2}(T_0)$} by \cite[Lemma 3.3]{camelier_affine_1997} (see also \cite{marmi_affine_2010} and \S~\ref{sec:prop:non_empty_affine} below), the following holds:
 
\begin{itemize}
\item If $\omega \in E_{ss}(T_0)$ then $T$ is $C^\infty$ conjugated to $T_0$, by \cite[Theorem 1]{cobo_piece-wise_2002}.
\item If $\omega \in E_{s}(T_0) \setminus E_{ss}(T_0)$ then $T$ is $C^1$, and not $C^2$, conjugated to $T_0$, by \cite[Theorem 1]{cobo_piece-wise_2002} and \cite[Theorem A]{liousse_echanges_2002}, 
\item If $g \geq 2$ and {$\omega \in E_{g - 1}(T_0) \setminus E_{g - 2}(T_0)$}, the AIET $T$ possesses a wandering interval, by \cite[Theorem 3.2]{marmi_affine_2010}. 
\end{itemize}
It is clear that in the first two cases, the AIET $T$ has no wandering intervals. 

\smallskip
The main results of this article (Theorem~\ref{thm:topconjugacy} and Theorem~\ref{thm:regularity} stated in \S~\ref{sec:main}) imply the following:
\begin{theorem*}
 Under full measure conditions on $T_0$, if the log-slope vector $\omega$ belongs to $E_{cs}(T_0) \setminus E_{s}(T_0)$, then $T$ is $C^0$ but not $C^1$-conjugate to $T_0$. 
\end{theorem*}
\noindent In particular, $T$ as above does not admit wandering intervals. Moreover, it will follow from Theorem \ref{thm:regularity} that, in this case, the unique invariant measure of $T$ is singular with respect to the Lebesgue measure. To prove the theorem above (in the form of Theorems~\ref{thm:topconjugacy} and \ref{thm:regularity}), we introduce a full measure condition in the space of IETs (see Definitions~\ref{def:BC_condition} and \ref{def:HS_Condition} and Proposition \ref{prop:fullmeasure}) which allows us to control the behaviour of the Zorich cocycle when restricted to $E_{cs}(T_0)$  and then studying Birkhoff sums of the piecewise constant function associated to the log-slope vector. 

The existence of a topological conjugacy (namely Theorem~\ref{thm:topconjugacy}, proved in \S~\ref{sec:wandering}) generalizes to full measure a result for periodic type IETs proved in the setting of substitutions by X. Bressaud, P. Hubert and A. Maass in \cite{bressaud_deviation_2014}. Let us point out that the absence of wandering intervals might also be inferred from the deep dynamical dichotomy proved in recent work \cite{ghazouani_priori_2021} by S. Ghazouani and the second author, but this would require assuming a much more subtle and technical Diophantine-like condition (see Definition 3.3.4 in \cite{ghazouani_priori_2021}), while the proof we provide here is simpler and self-contained.  
Furthermore, we prove a result about Birkhoff sums of piecewise constant functions in the  space over IETs, which is of independent interest (see Proposition~\ref{prop:boundedseq} and in particular, Corollary~\ref{cor:boundedseq}). 
The singularity of the conjugacy (Theorem~\ref{thm:regularity}) can also be deduced from the work of M. Cobo \cite{cobo_piece-wise_2002}, which is in turn based on work by W. Veech \cite{veech_metric_1984} (see Appendix~\ref{app:Cobo}). We provide an independent proof in \S~\ref{sc:singularity}. 

\smallskip
We conclude by commenting on the interest of this result from the point of view of the study of generalized interval exchange transformations (GIETs), in the light of recent and ongoing work that evidences the crucial role played by AIETs in the study of GIETs. 

 In \cite{ghazouani_priori_2021}, S. Ghazouani and the second author proved that to a given GIET, under a full measure condition on the IET rotation number, given by an IET $T_0$, one can associate an AIET called the (unstable) \emph{shadow}. When the log-slope $\omega$ of this shadow is non-zero, one expects wandering intervals and the lack of a topological conjugacy, a result which for now was proved in genus two, see \cite{ghazouani_priori_2021}. 

On the other hand, when $\omega$ is zero, and the \emph{boundary} of the GIET (an invariant defined by S. Marmi, P. Moussa, and J.-C. Yoccoz in \cite{marmi_linearization_2012}) is zero, it is shown in \cite{ghazouani_priori_2021} that one can prove, in the spirit of M. Herman's work \cite{herman_sur_1979} on circle diffeomorphisms, the existence of a differentiable conjugacy between the GIET and its IET model (see also \cite{marmi_linearization_2012}, and \cite{ghazouani_local_2021} for local results describing $\mathcal{C}^r$-conjugacy classes of IETs for $r\geq 2$ and $r=1$ respectively). The result of this paper indicates that the assumption that the boundary is zero is necessary to have a non-singular conjugacy. 

The study of GIETs which have non-zero boundary (but total non-linearity zero) is undertaken in \cite{berk_rigidity_2022} by P. Berk and the first author. For those, when the log-slope $\omega$ of the (unstable) shadow of \cite{ghazouani_local_2021} vanishes, one can define a finer notion of (central) shadow, which allows recovering rigidity results that naturally generalize the known rigidity results for PL-circle diffeomorphisms. The absence of wandering intervals for AIETs that we prove in this paper (namely Theorem~\ref{thm:topconjugacy}) then provides the leverage to show the absence of wandering intervals also for the GIETs in the considered class.

\section{Background material and notations}\label{sec:background}
Let us start by recalling some of the basic notions and properties related to IETs and introduce some notations. The objects we will consider are now classical; we refer the interested reader to \cite{viana_ergodic_2006}, \cite{yoccoz_interval_2010} for a complete introduction to the subject as well as for proofs and additional details. 

\subsection{Standard and affine interval exchange transformations}\label{sec:AIETs}

A \emph{standard interval exchange transformation}, or simply an \emph{interval exchange transformation} (IET), is a bijective right-continuous piecewise translation of an interval with a finite number of discontinuities. More precisely, given a compact interval $I \subset \R$, we say that a bijection $T: I \to I$ is an IET on $d \geq 2$ intervals if there exists a partition of $I$ on $d$ disjoint left-closed and right-open subintervals of $I$ such that $T$ is a translation when restricted to each of the intervals on the partition.  An IET with $d \geq 2$ intervals can be described by the way the intervals are exchanged and their lengths. For this, we fix a finite alphabet $\A$ with $d$ elements and consider pairs $(\pi_0,\pi_1)$ of bijections $\pi_0,\pi_1:\mathcal A\to\{1,\dots,d\}$ to denote the order of the intervals before and after the exchange. We always assume that the datum $(\pi_0,\pi_1)$ is \emph{irreducible}, i.e.
$$\pi_1\circ\pi_0^{-1}(\{1,\ldots,k\})\!=\!\{1,\ldots,k\}\Rightarrow k=d. $$
 The class of \emph{irreducible} IETs, i.e.~IETs with irreducible data $(\pi_0,\pi_1)$ and $d \geq 2$ intervals, can then be parametrized by the set $\mathscr{I}_\A^+ = \IETSpace,$ where $\PermSpace $ denotes the set of irreducible pairs $ (\pi_0,\pi_1) $ of bijections of $d$ symbols, and the set of \emph{normalized IETs} on $d$ intervals, that is, IETs defined on the unit interval $I = [0, 1)$, by $\mathscr{I}_\A = \IETSpaceNorm$, where 
\[\Simplex = \left\{\lambda\in\R_+^{\A} \left| \ |\lambda|_1 =1 \right.\right\}.\]
We endow $\IETSpace$ and $\IETSpaceNorm$ with the product measure $d\pi \times \textup{Leb}$, where $d\pi$ denotes the counting measure in $\PermSpace$. 
{We say that a property holds for \emph{almost every} IET on $d$ intervals if it holds for almost every point of $\IETSpace$ with respect to this product measure.}

\subsubsection*{Affine interval exchange transformations.} An \emph{affine interval exchange transformation} (AIET) is a bijective right-continuous piecewise affine of an interval with a finite number of discontinuities and having positive slope on each continuity interval. Similarly to IETs, we encode AIETs using the order in which intervals are exchanged, their lengths, and the logarithm of the slope on each continuity interval (the use of log-slopes instead of slopes will be justified later on, see \eqref{eq:log_slopes_height_cocycle}).  Thus, $ e^{\omega_\alpha} $ is, by definition, the slope of the restriction of $T$ to the interval indexed by $\alpha\in \A$. Notice that if this interval has length $\eta_\alpha>0$, its image by $T$ has length $\eta_\alpha e^{\omega_\alpha}$. 

We can parametrize the set of AIETs on $d$ intervals by
\[ \mathscr{A}_\A^+ = \left\{ (\pi, \eta, \omega) \in \IETSpace \times \R^\A \,\left|\, \sum_{\alpha \in \A} \eta_\alpha e^{\omega_\alpha} = \sum_{\alpha \in \A} \eta_\alpha \right.\right\}.\] 
The last condition guarantees that the sum of the lengths of the images under $T$ of the subintervals is the same as the domain length. We parametrize the set of \emph{normalized AIETs} analogously by a set $\mathscr{A}_\A \subset \IETSpaceNorm \times \Delta_\A$. 

\subsection{Rauzy-Veech and Zorich induction}\label{sec:RV}
A classical induction procedure for IETs, known as the \emph{Rauzy-Veech induction}, as well as its subsequent normalizations and accelerations, are well known to be extremely useful in studying IETs (as well as AIETs and GIETs). We recall some basic definitions and notations in this section and refer the reader to \cite{viana_ergodic_2006} or \cite{yoccoz_interval_2010} for a detailed introduction.

\subsubsection*{Rauzy-Veech induction algorithm.}
The Rauzy-Veech induction associates to almost every interval exchange transformation (IET) another IET, with the same number of intervals, by inducing the initial transformation into an appropriate subinterval. The subinterval is chosen according to the \emph{type} of the IET, which encodes whether the `last' interval in the partition, i.e., $I_{\pi_0^{-1}(d)}$, is longer or smaller than the interval going to the last position after applying the transformation, i.e., $I_{\pi_1^{-1}(d)}$. This procedure can be iterated infinitely many times if and only if the IET satisfies \emph{Keane's condition}. By \cite{keane_interval_1975}, any IET satisfying Keane's condition is minimal. The Rauzy-Veech induction defines an oriented graph structure in $\PermSpace$, called the \emph{Rauzy graph}. Each connected component in this graph is called a \emph{Rauzy class}.  The infinite path in the Rauzy graph defined by an IET satisfying Keane's condition is called \emph{combinatorial rotation number}.

We denote the set of \emph{IETs verifying Keane's condition} by $\Dom \subset \IETSpace$, and the \emph{Rauzy-Veech induction} and the \emph{Zorich acceleration} by $$\RV: \Dom \to \Dom,\quad\quad \ZorichMap: \Dom \to \Dom,$$
respectively. The map $\ZorichMap$ is defined as $\ZorichMap\IET = \RV^{z\IET}\IET$ {where the measurable map $z: \Dom \to \N$ is defined so that ${z\IET}$ is the largest integer such that $\IET, \RV \IET, \dots, \RV^{z\IET-1}\IET$ all have the same type.} 

\subsubsection*{Notations} Given an IET $(\pi, \lambda) \in \Dom$, we denote its \emph{type} by $\epsilon(\pi, \lambda) \in \{0, 1\}$, its \emph{winner} (resp. \emph{loser}) \emph{symbol} by $\alpha_\epsilon(\pi, \lambda)$ (resp. $\alpha_{1 - \epsilon(\pi, \lambda)}$). { 
Assume that $T_0 = \IET$ verifies Keane's condition so that $\RV^n(T_0)$ is defined for any $n\in \N$. We denote the  \emph{combinatorial rotation number} of $\IET$ by $\gamma\IET$.  
For any $n \geq 0$, we denote }
\begin{equation*}
\begin{aligned}
&T^{(n)} = \big(\pi^{(n)}, \lambda^{(n)}\big) = \RV^n(T_0), & \text{ orbit of } \IET \text{ by } \RV, \\
& I^{(n)}(T_0), & \text{ domain of definition of } \RV^n(T_0),\\
& I^{(n)}_\alpha(T_0), & \text{ intervals exchanged by } \RV^n(T_0),\\
& q^{(n)}(T_0) = (q^{(n)}_\alpha(T_0))_{\alpha \in \A}, & \text{ the return time of } I^{(n)}_\alpha \text{ to } I^{(n)} \text{ by } T_0.
\end{aligned}
\end{equation*}
If there is no risk of confusion, we will omit the explicit dependence on $T_0$ in all of the above notations. 

\subsubsection*{Dynamical partitions}
Given an IET $T_0 = \IET$ verifying Keane's condition, we can associate a sequence of \textit{dynamical partitions} and \emph{Rohlin towers} as follows. We define the \emph{dynamical partition} $\mathcal{P}^{(n)}$ of $I$ \emph{at level} $n$ as
$$ \mathcal{P}^{(n)} := \bigcup_{\alpha \in \A} {\mathcal{P}^{(n)}_\alpha}, \qquad \text{where}\quad \mathcal{P}^{(n)}_\alpha = \big\{ I_\alpha^{(n)}, T\big(I_\alpha^{(n)}\big), \cdots, T^{q_\alpha^n - 1}\big(I_\alpha^{(n)}\big)\big\}. $$
One can verify that $\mathcal{P}^{(n)}$ is a partition of $[0,1)$ into subintervals and that, for each $\alpha \in \A$, the collection $\mathcal{P}^{(n)}_\alpha$ is a Rohlin tower of height $q_\alpha^n$. Notice that if $n>m$, then $\mathcal{P}^{(n)}$ is a refinement of $\mathcal{P}^{(m)}$. 

\subsubsection*{Zorich cocycle}
In the following, for any $F: X \to X$, $\phi: X \to GL(d, \Z)$ and $n > m \geq 0$, we denote 
\[ \phi_{m, n}(x) = \phi(F^{n - 1}(x)) \cdot \dots \cdot \phi(F^m(x)).\]
The length vector and the return times of the iterates of an IET by the Zorich map can be described via a cocycle 
\[B: \Dom \to SL(\A, \Z),\]
that we obtain as a proper acceleration of the cocycle
\[\Function{A}{\Dom}{SL(\A, \Z)}{(\pi, \lambda)}{\textup{Id} + E_{\alpha_{\epsilon(\pi, \lambda)}, \alpha_{1 - \epsilon(\pi, \lambda)}}},\]
which encodes the change of the length vector after one step of Rauzy-Veech induction. More precisely, for any $n > m \geq 0$, the cocycles $A^{-1}$ and $A^T$ verify 
\begin{align}
\label{eq:length_prop_cocycle}
& \lambda^{(n)} = \LRC{m}{n}\IET\lambda^{(m)}, \\
\label{eq:height_prop_cocycle}
& q^{(n)} = \HRC{m}{n}\IET q^{(m)},
\end{align}
where $q^{(0)} = \overline{1} \in \R_+^\A$ is the vector whose coordinates are all equal to $1$. 

Defining $B\IET = A\IET\dots A(\pi^{(z\IET - 1)}, \lambda^{(z\IET - 1)}),$ the accelerated cocycles $B^{-1}$ and $B^T$ verify analogous properties with respect to the Zorich map. 

The cocycle $B^T$ is called the \emph{Zorich cocycle} or \emph{Kontsevich-Zorich cocycle}. Since $B^{-1}$ is also sometimes referred to as the Zorich cocycle, and in view of \eqref{eq:length_prop_cocycle}, \eqref{eq:height_prop_cocycle}, to avoid any possible confusion, we will refer to $B^{-1}$ as the \emph{length cocycle} and to $B^T$ as the \emph{height cocycle}.

{
\subsubsection*{Dynamical interpretation of entries}\label{sc:incidence_matrices}

The matrices $\ZC{m}{n}$ (and consequently their accelerations) have the following dynamical interpretation.} The $\alpha\beta$-th entry of the incidence matrix $\ZC{m}{n}$ is the number of times the orbit by $T^{(m)}$ of any $x \in I^{(n)}_{\alpha}$ visits $I^{(m)}_{\beta}$ up to its first return to $I^{(n)}$. The incidence matrix entries also have an interpretation in terms of Rohlin towers. In fact, they describe how the Rohlin towers in the dynamical partition $\mathcal{P}^{(n)}$ can be obtained by \emph{cutting and stacking} 
Rohlin towers of $\mathcal{P}^{(m)}$. More precisely, for any $\alpha$, the Rohlin tower $\mathcal{P}^{(n)}_\alpha$ is obtained by stacking \emph{subtowers} of the Rohlin towers $\mathcal{P}^{(m)}_\beta$, $\beta \in \A$ (namely, sets of the form $\{ T^k J \mid 0 \leq k <q^{m}_\beta\}$ for some subinterval $J\subset I_\beta^{(m)}$). Indeed, the $\alpha\beta$-th entry of the incidence matrix $\ZC{m}{n}$ is the number of subtowers of $\mathcal{P}^{(m)}_\beta$ inside $\mathcal{P}^{(n)}_\alpha$. It follows that $\mathcal{P}^{(n)}_\alpha$ is made by stacking exactly $\sum_{\beta \in \A} (\ZC{m}{n})_{\alpha\beta}$ subtowers of Rohlin towers of $\mathcal{P}^{(m)}$. 

\subsection{Rauzy-Veech induction for AIETs}
The Rauzy-Veech induction and the Zorich acceleration extend naturally to the space of AIETs, as well as all the notions introduced above in the IET setting, such as combinatorial rotation number, dynamical partitions, incidence matrices, etc. Given an AIET $T = (\pi, \eta, \omega)$ satisfying Keane's condition we denote its orbit under $\RV$ by $$T^{(n)} = (\pi^{(n)}, \eta^{(n)}, \omega^{(n)}) = \RV^n (\pi, \eta, \omega), \qquad n \in \N.$$ 
For the sake of simplicity, we will use the notations introduced in the IET setting to denote the intervals of definition and the return times of iterates by $\RV$ of an AIET.

 Let us point out that the incidence matrices $ \ZC{m}{n}$ depend only on the combinatorial rotation number. In particular, given an AIET $T$ and an IET $T_0$, both satisfying Keane's condition and such that $\gamma(T) = \gamma(T_0)$, the incidence matrices of $T$ and $T_0$ coincide. 

In the context of AIETs, the height cocycle verifies an additional property of fundamental importance to us: the change in the log-slope vector of $\RV$ iterates of an AIET is described by the height cocycle. More precisely, given an AIET $T = (\pi, \eta, \omega)$ satisfying Keane's condition, for any $n \geq m \geq 0$,
\begin{equation}
\label{eq:log_slopes_height_cocycle}
\omega^{(n)} = \ZC{m}{n}(\pi, \eta, \omega)\omega^{(m)}.
\end{equation}

\subsection{Oseledet's filtration}\label{sec:filtration}
As mentioned before, the \emph{normalized version of $\ZorichMap$}, which is defined in the subset $\DomNorm \subset \IETSpaceNorm$ of normalized IETs satisfying Keane's condition and that we denote by $$\ZorichNorm: \DomNorm \to \DomNorm,$$ admits a unique invariant probability measure $\mu_{\ZorichNorm}$ equivalent to the Lebesgue measure on $\DomNorm$. Moreover, the height and length cocycles, $B^T$ and $B^{-1}$ are integrable with respect to this invariant measure, and thus they admit invariant Oseledet's filtrations
\begin{gather*}
E_s\IET \subset E_{cs} \IET \subset \R^{\A},\\
F_s\IET \subset F_{cs} \IET \subset \R^{\A},
\end{gather*}
for a.e. $\IET \in \DomNorm$, respectively. 

With these notations, the sets $E_s\IET$, $E_{cs}\IET \,\setminus\, E_s\IET$ and $\R^\A \,\setminus\,E_{cs}\IET$, correspond to the set of vectors with negative, zero and positive Lyapunov exponents for the cocycle $B^T$, respectively. That is, for a.e. $\IET \in \DomNorm$ and for every $v \in \R^\A$, the limit 
\[ \theta(\pi, \lambda, v)= \lim_{n \to +\infty} \frac{\log |\HC{0}{n}\IET v|_1}{n}\]
exists and verifies 
\[ \left\{ \begin{array}{lcl} 
\theta(\pi, \lambda, v) < 0 & \text{if} & v \in E_s\IET,\\
\theta(\pi, \lambda, v) = 0 & \text{if} & v \in E_{cs}\IET \,\setminus\, E_s\IET, \\
\theta(\pi, \lambda, v) > 0 & \text{if} & v \in \R^\A \,\setminus\,E_{cs}\IET. \end{array}\right.\]
Analogous properties hold for the cocycle $B^{-1}$ and its associated splitting. 

We point out that the dimension of these vector spaces depends only on the permutation $\pi$. Indeed, denoting $\Omega_\pi: \R^\A \to \R^\A,$ where 
\begin{equation}
\label{eq:exchange_matrix}
(\Omega_\pi)_{\alpha, \beta} = \left\{ \begin{array}{cl} +1 & \text{if } \pi_1(\alpha) > \pi_1(\beta) \text{ and } \pi_0(\alpha) < \pi_0(\beta), \\
-1 & \text{if } \pi_1(\alpha) < \pi_1(\beta) \text{ and } \pi_0(\alpha) > \pi_0(\beta), \\
0 & \text{in other cases,} 
\end{array}\right.
\end{equation}
we have that  
\begin{equation}
\label{eq:dimension_oseledets}
\begin{aligned}
& \dim(E_s\IET) = g, \quad \dim(E_{cs}\IET) = d - g, { \quad \text{where} \ g := \frac{d - \dim(\textup{Ker}(\Omega_\pi))}{2},}
\end{aligned}
\end{equation}
for a.e. $\IET \in \DomNorm$. 

\section{Statements of the results}
In this section, we state our main results. Let us start by recalling some of the existent results concerning the semi-conjugacies of AIET to IETs. Recall that the \emph{combinatorial rotation number} of an IET that satisfies Keane's condition is, by definition, the infinite path in the Rauzy graph produced by iterating the Rauzy-Veech induction procedure. 

\subsection{Semiconjugacies of AIETs to IETs}
\label{sec:prop:non_empty_affine}

An infinite Rauzy path $\gamma$ is said to be \emph{$\infty$-complete} if every symbol in $\A$ appears infinitely many times as a winner symbol in $\gamma$. It is well-known that any IET satisfying Keane's condition defines an $\infty$-complete Rauzy path in the Rauzy graph and, conversely, any $\infty$-complete Rauzy path determines a unique normalized IET (for a proof, see e.g. \cite[Section 7]{yoccoz_echanges_2005}). 

Given an infinite path $\gamma$ in the Rauzy graph and $\omega \in \R^\A$, we denote by $\Saff$ the set of normalized AIETs with log-slope $\omega$ and combinatorial rotation number $\gamma.$ If a path $\gamma$ is $\infty$-complete, maps in $\Saff$ are semi-conjugated to the unique IET whose rotation number is equal to $\gamma$. More precisely, we have the following.

\begin{proposition}[Proposition 7 in \cite{yoccoz_echanges_2005}]
\label{prop:semi-conjugacy}
Let $T_0$ be an IET such that $\gamma(T_0)$ is $\infty$-complete and let $\omega \in \R^\A$. Then, any $T \in \textup{Aff}(\gamma(T_0), \omega)$ is semi-conjugated to $T_0$ via an increasing surjective map $h: [0, 1) \to [0, 1)$, satisfying $T_0 \circ h = h \circ T$. Moreover, if $T$ has no wandering intervals, then $h$ defines a conjugacy between $T_0$ and $T$.
\end{proposition}

However, not all choices of $\gamma$ and $\omega$ are compatible.

\begin{proposition}[Proposition 2.3 in \cite{marmi_affine_2010}]
\label{prop:non_empty_affine}
Let $T_0 = \IET$ be an IET such that $\gamma(T_0)$ is $\infty$-complete and let $\omega \in \R^\A$. Then, $\textup{Aff}(\gamma(T_0), \omega) \neq \emptyset$ if and only if $\langle \omega, \lambda \rangle = 0$.
\end{proposition}

\subsection{Main results}\label{sec:main}
We now state the main results of this article. 

\begin{theorem}
\label{thm:topconjugacy}
For almost every IET $T_0$ and for any $\omega \in E_{cs}(T_0) \,\setminus\, E_s(T_0)$, any AIET $T \in \textup{Aff}(\gamma(T_0), \omega)$ is topologically conjugated to $T_0$. 
\end{theorem}
{\noindent A special case of this theorem, namely, the same result for the (measure zero) set of IETs whose combinatorial rotation number is periodic (also known as \emph{periodic-type} IETs), was proved, in the setting and language of substitutions, in \cite{bressaud_deviation_2014}.} The proof of Theorem~\ref{thm:topconjugacy} is presented in \S\ref{sec:topconj}.
 
We also prove the following.
\begin{theorem}
\label{thm:regularity}
For almost every IET $T_0$ and for any $\omega \in E_{cs}(T_0) \,\setminus\, E_s(T_0)$, any AIET $T \in \textup{Aff}(\gamma(T_0), \omega)$ is uniquely ergodic and its unique invariant probability measure is singular with respect to the Lebesgue measure.
\end{theorem}
\noindent {Theorem~\ref{thm:regularity} can be deduced combining Theorem~\ref{thm:topconjugacy} with a result proved by M. Cobo \cite[Theorem 1]{cobo_piece-wise_2002}, which in turns { rely  on results by}  W. Veech~\cite{veech_gauss_1982} (see Appendix~\ref{app:Cobo}, where the result and the deduction are presented). The proof we give in this paper, which is given in \S~\ref{sc:singularity}, { has the advantage of being self-contained and perhaps more transparent.}} 

Let us mention that by the \emph{duality} of the heights and lengths cocycles it follows that 
\begin{equation}\label{rk:centralinclusion}
E_{cs}(T_0) \subset \lambda^\bot,
\end{equation} for a.e. IET $T_0$ (we refer the interested reader to \cite{zorich_deviation_1997} for a precise definition of dual cocycle and to \cite[pages 384-385]{cobo_piece-wise_2002} for a proof of this fact). Thus, if $\pi \in \PermSpace$ is such that $\textup{Ker}(\Omega_\pi) \neq \{0\}$, where $\Omega_\pi$ is the matrix given by \eqref{eq:exchange_matrix}, it follows from \eqref{eq:dimension_oseledets} and Proposition \ref{prop:non_empty_affine} that the set $\textup{Aff}(\gamma(T_0), \omega)$ in Theorems \ref{thm:topconjugacy} and \ref{thm:regularity} is non-empty. 

\subsection{The full measure Diophantine-type conditions}\label{sec:fullmeasure}
Let us now state explicitly the generic condition satisfied by an IET $T_0$ for Theorem \ref{thm:topconjugacy} to hold. This condition is an example of a \emph{Diophantine-type condition} on an IET rotation number. 
\begin{definition}[The BC condition]
\label{def:BC_condition}
We say that an IET $T_0 = \IET$ satisfies the \emph{Bounded Central} Condition (or, for short, the BC Condition) if it verifies Keane's condition, it is Oseledets generic, $\gamma(T_0)$ is $\infty$-complete, and there exists a sequence $(n_k)_{k\in \mathbb{N}} \subset \N$ verifying the following:
\begin{enumerate}[(i)]
\item \label{cond:positive_matrix} There exists $N \in \N$ and a constant $K > 0$ such that
\[ 1 \leq \ZC{n_k}{n_k + N}\IET_{\alpha\beta} \leq K, \qquad \forall \alpha, \beta \in \A, \,\forall k\in \mathbb{N}. \]
\item \label{cond:bounded_central_stable} There exists a constant $V > 0$ such that 
\[ \left \|\ZC{0}{n_k}\IET\mid_{E_{cs}(\lambda, \pi)} \right \| \leq V, \qquad \forall k\in \mathbb{N}. \]
\end{enumerate}
\end{definition}
\noindent As we shall see, this is a full measure condition in the space of IETs (see Proposition~\ref{prop:fullmeasure} below).

\smallskip
For the proof of Theorem~\ref{thm:regularity}, it is also useful to introduce the following condition.
\begin{definition}[HS Condition]
\label{def:HS_Condition}
We say that an IET $T_0 = \IET$ satisfies the \emph{high singularities} Condition (or, for short, the HS Condition) if it verifies Keane's condition and there exist $C > 0$ and a sequence $(n_k)_{k\in \mathbb{N}} \subset \N$ verifying the following:
\begin{enumerate}[(i)]
\item\label{cond:balanced_heights} $\max_{\alpha, \beta \in \A} \frac{q^{(n_k)}_\alpha}{q^{(n_k)}_\beta} < C.$
\item \label{cond:continuous_iterates} $T^i\mid_{I^{(n_k)}}$ is continuous for $0 \leq i \leq \tfrac{1}{4} \max q^{(n_k)}_\alpha$.
\end{enumerate}
\end{definition}

\begin{proposition}
\label{prop:fullmeasure}
Almost every irreducible IET satisfies the BC and HS Conditions.
\end{proposition}
\noindent For the sake of clarity of exposition, we postpone the proof of the above proposition to \S\ref{sc:fullmeasureproof} as it requires the introduction of the natural extension of the Zorich renormalization as well as several definitions and notations that will not appear anywhere else in the article. 

\smallskip 
We now prove Theorems \ref{thm:topconjugacy} and \ref{thm:regularity}, respectively in \S\ref{sc:proof_topconjugacy} and \S~\ref{sc:singularity}. 

\section{Existence of a topological conjugacy}\label{sec:topconj}

In this section, we prove Theorem \ref{thm:topconjugacy}, by showing the existence of a topological conjugacy between an AIET and IET having the same combinatorial rotation number, under the BC condition (see Definition \ref{def:BC_condition}) on the IET. 

\subsection{Wandering intervals and Birkhoff sums}

Let us recall that for $T$ and $T_0$ as in Theorem \ref{thm:topconjugacy}, and under a full measure condition on $T_0$, the assumption $\gamma(T)=\gamma(T_0)$ automatically yields a semi-conjugacy $h$ between $T$ and $T_0$. Moreover, this semi-conjugacy is indeed a conjugacy if the map $T$ has no wandering intervals (see Proposition \ref{prop:semi-conjugacy}).

The main criterion we will use to exclude the presence of wandering intervals on a given AIET is stated in Lemma~\ref{lemma:wanderingintervalscriterium}, and it is due to M. Cobo \cite{cobo_piece-wise_2002} (see also \cite{bressaud_persistence_2010}). This criterion reduces the question of the existence of wandering intervals to a study of Birkhoff sums of the log-slope vector of the AIET. Let us first introduce some notation. 

\noindent Let $f:[0,1]\to \mathbb{R}$ be a real-valued function and let $T$ be an AIET. For each $n\in \mathbb{Z}$, we define the $n$-th \emph{Birkhoff sum} of $f$ over $T$ by
\begin{equation}\label{def:BS}
S_n^Tf:= 
\begin{cases} \sum_{j=0}^{n-1}f \circ T^j , & \text{if}\ n>0,\\
0 , & \text{if}\ n=0,\\
\sum_{j=n}^{-1}f \circ T^{-j} , & \text{if}\ n<0.\\
\end{cases}
\end{equation}
If there is no risk of confusion concerning the transformation being considered, we will denote the Birkhoff sums simply by $S_n f$. The definition of the Birkhoff sums $S_nf$, for $n\leq 0$, is given so that $(S_nf)_{n\in\mathbb{Z}}$ is a $\mathbb{Z}$-additive cocycle, i.e. satisfies 

\begin{equation}\label{eq:cocyclerel}
S_{n+m}\, f = S_n\, f + S_m f\circ T^n,\qquad \text{ for\ all \ } n,m \in \mathbb{Z}.
\end{equation}

The space of piecewise-constant real-valued functions, which are continuous on each of the continuity intervals of $T$, can be identified with a vector in $\R^\A$. Indeed, given a vector $\omega \in \R^\A$, we associate the piecewise constant function $f_{T,\omega} : = I \to \R$ given by
\begin{equation}\label{def:fv}
f_{T,\omega} (x):= \omega_\alpha, \qquad \text{if}\quad x\in I_\alpha \text{ for some }\alpha \in \A.
\end{equation}
We will also write simply $f_\omega $ instead than $f_{T,\omega}$ when the dependence on $T$ is clear. 
In view of the following criterium, the existence of wandering intervals for an AIET with log-slope vector $\omega$ can be reduced to the study of Birkhoff sums of the function $f_\omega$.

\begin{lemma}[Wandering intervals via Birkhoff sums, \cite{cobo_piece-wise_2002}]\label{lemma:wanderingintervalscriterium}
An AIET $T$ with log-slope vector $\omega$ has a wandering interval if and only if there exists a point $x_0\in [0,1]$ such that
$$
\sum_{n\geq 1} e^{S_{n} f_\omega(x_0)}<+\infty, \qquad \sum_{n\leq 0} e^{S_{n} f_\omega(x_0)}<+\infty.
$$
\end{lemma}
The proof of this lemma can be found in \cite[pp. 392-393]{cobo_piece-wise_2002}. This criterion has also been used in \cite{bressaud_persistence_2010}, and \cite{marmi_affine_2010}. The result we prove in order to exploit this Lemma is the following.

\begin{proposition}\label{prop:boundedseq}
Let $T_0$ be an IET satisfying the BC condition and let $\omega \in E_{cs}(T_0)$. Then, for any AIET $T \in \textup{Aff}(\gamma(T_0), \omega)$ and for any $x\in [0,1]$ there exists a sequence $(m_k)_{k\in \mathbb{Z}}$ with $|m_k|\to \infty$ as $|k|\to \infty$ such that 
\begin{equation}\label{eq:boundedBS}
\sup_{k\in \mathbb{Z}}|S_{m_k}^T f_{T,\omega}(x)| < +\infty.
\end{equation}
\end{proposition}
In particular, this Proposition gives a result for Birkhoff sums over an IET $T_0$ of a piecewise constant function $f_{T_0,\omega } $ associated to a vector $\omega\in E_{cs}(T_0) $, which we believe is of independent interest, namely:
\begin{corollary}\label{cor:boundedseq}
For $T_0$ and $\omega$ as in Proposition~\ref{prop:boundedseq},  for any $x\in [0,1]$ there exists a sequence $(m_k)_{k\in \mathbb{Z}}$ with $|m_k|\to \infty$ as $|k|\to \infty$ such that 
$$
\sup_{k\in \mathbb{Z}}|S_{m_k}^{T_0} f_{T_0,\omega}(x)| < +\infty.
$$
\end{corollary}

The proof of Proposition \ref{prop:boundedseq} exploits the decomposition of Birkhoff sums into special Birkhoff sums, which are building blocks controlled by the cocycle matrices produced by Rauzy-Veech induction. In \S~\ref{sec:sums}, we recall their definition and some basic properties.

Before proving Proposition \ref{prop:boundedseq}, whose proof we postpone to \S\ref{sec:proofboundedseq}, let us show how to use it to prove Theorem \ref{thm:topconjugacy}. 

\subsection{{Proof of absence of wandering intervals}}
\label{sc:proof_topconjugacy}
In this section, we prove Theorem~\ref{thm:topconjugacy} by showing the absence of wandering intervals under the theorem's assumptions. 

\begin{proof}[Proof of Theorem~\ref{thm:topconjugacy}]
Since $\gamma(T)=\gamma(T_0)$, by Proposition \ref{prop:semi-conjugacy}, there exists a semi-conjugacy, i.e.~a surjective continous increasing function $h:[0, 1] \to [0,1]$ such that $h\circ T=T_0\circ h$. To show that $h$ is a homeomorphism and, therefore, a topological conjugacy, it is enough to show that it has no wandering intervals. This follows by Lemma~\ref{lemma:wanderingintervalscriterium}, since, by Proposition~\ref{prop:boundedseq}, for any $x\in [0,1]$, there exists $C>0$ and $(m_k)_{k\in \mathbb{Z}}$ such that 
\[\sup_{k\in \mathbb{Z}}|S_{m_k}^T f_\omega(x)| < C.\] Hence
$$
\sum_{n\geq 1}e^{S_{n}^T f_\omega(x_0)} \geq \sum_{k\geq 1} e^{S_{m_k}^T f_\omega(x_0)} \geq \sum_{k\geq 1} e^{-C} =+\infty. 
$$
Similarly, $\sum_{n\leq 0} S_n^T f_\omega(x_0) =+\infty$. Therefore, Lemma \ref{lemma:wanderingintervalscriterium} implies that $T$ has no wandering intervals.
\end{proof}

\subsection{Birkhoff sums and special Birkhoff sums.}\label{sec:sums}

Given an AIET $T$ and a function $f: I \to \R$, the Birkhoff sums $S_n f$, for $n\geq 0$, can be studied via renormalization exploiting the notion of \emph{special Birkhoff sums} that we now recall. For any $n \in \N$, the \emph{special Birkhoff sum of level $n$} is the function $\SBS{n}{f}: I^{(n)} \to \mathbb{R}$ obtained by \emph{inducing} $f$ over the first return map $T^{(n)}$. More precisely, 
\[ \SBS{n}{f}(x):= S_{q^{(n)}_\alpha}f(x) =\sum_{\ell=0}^{q^{(n)}_\alpha - 1} f \left(T^\ell( x)\right), \quad \text{if}\ x\in I^{(n)}_\alpha, \quad \text{for any } \alpha \in \A, \ n\in\mathbb{N}.\]
Thus one can think of $\SBS{n}{f}(x)$ as the Birkhoff sum of $f$ at $x$ \emph{along the Rohlin tower} of height $q^{(n)}_\alpha$ over $I^{(n)}_\alpha$. 
Given the $n$-th special Birkhoff sum $\SBS{n}{f}$, we can build $\SBS{n + 1}{f}$ from $\SBS{n}{f}$ and $T^{(n)}$ as
\begin{equation}\label{SBSdecomp}
\SBS{n + 1}{f} (x) = \sum_{k=0}^{a_\alpha^{(n)}} \SBS{n}{f} \big((T^{(n)})^k(x)\big), \qquad \text{for\ any}\ x\in I^{(n + 1)}_\alpha,
\end{equation}
where $a_\alpha^{(n)} = \sum_{\beta \in \A}(\ZC{n}{n + 1})_{\alpha\beta}$. Equation \eqref{SBSdecomp} follows from the relation between Rohlin towers of the partitions $\mathcal{P}^{(n + 1)}$ and $\mathcal{P}^{(n)}$ described in \S\ref{sc:incidence_matrices}.

Hence, if $f = f_\omega$ for some $\omega \in \R^A$, it follows from the definition of special Birkhoff sums that 
\begin{equation}
\label{SBS_cocycle_relation}
\SBS{n}{f_\omega}(x) = \ZC{0}{n}\omega.
\end{equation}

\subsection{Proof of {  the Birkhoff sums upper bound (Proposition~\ref{prop:boundedseq})}}
\label{sec:proofboundedseq}
We are now ready to prove Proposition~\ref{prop:boundedseq}. The argument behind the proof provides a generalization of the key idea exploited in \cite{bressaud_deviation_2014} to prove the analogous result in the special case of periodic combinatorial rotation numbers. While the proof in \cite{bressaud_deviation_2014} is written in the language of substitutions and prefix-suffix decompositions, our proof below uses simply the decomposition of Birkhoff sums in special Birkhoff sums and the BC condition. The key idea is that for \emph{any} point $x\in [0,1]$ it is possible to find times $(m_k)_{k\in \mathbb{Z}}$ such that $S_{m_k} f_\omega (x)$ can be decomposed into a bounded number of special Birkhoff sums. 

\begin{proof}[Proof of Proposition~\ref{prop:boundedseq}]
Let $(n_k)_{k\in \mathbb{N}}$ be the sequence of induction times given by the BC condition (see Definition \ref{def:BC_condition}). For the sake of simplicity and clarity of exposition, let us denote
\[T_k = T^{(n_k)}, \quad {I}^k:=I^{(n_k)}, \quad {Z}^k_\alpha:={Z}^{(n_k)}_\alpha, \quad {q}^k_\alpha:={q}^{(n_k)}_\alpha, \quad \mathcal{P}^k = \mathcal{P}^{(n_k)},\]
\[{I}^{k_+}:=I^{(n_k + N)}, \quad {Z}^{k_+}_\alpha:={Z}^{(n_k + N)}_\alpha, \quad {q}^{k_+}_\alpha:={q}^{(n_k + N)}_\alpha, \quad \mathcal{P}^{k_+} = \mathcal{P}^{(n_k + N)},\]
for any $k \in \N$ and $\alpha \in \A$, where $N$ is given by the BC condition. Recall that, for any $n \in \N$, $I^{(n)}$ stands for the inducing subinterval of the $n$-th step of Rauzy-Veech induction while ${Z}^{(n)}_\alpha$, $q^{(n)}_\alpha$ and $\mathcal{P}^{(n)}$ denote the corresponding dynamical Rohlin towers, their heights, and the associated dynamical partition, respectively.

Let $x \in [0, 1]$ be fixed. We will define a sequence $(m_{k})_{k \in \Z}$ such that \eqref{eq:boundedBS} holds. Fix $k \in \N$. Let $\alpha, \beta\in \mathcal{A}$ be such that $x$ belongs to the towers $Z^k_\alpha$ and $Z^{k_+}_\beta$, respectively, and let $0\leq i< q^{k}_\alpha$ and $0\leq j< q^{k_+}_\beta $ be the indexes of the floors or $Z^k_\alpha$ and $Z^{k_+}_\beta$ respectively which contains $x$, i.e.~such that
$$
x\in T^{i} (I^k_{\alpha}), \qquad x\in T^{j} (I^{k_+}_{\beta}). 
$$
Let also $\beta^-$ and $\beta^+$ the the indexes of the dynamical towers of the partition $\mathcal{P}^{k_+}$ \emph{before} and \emph{after} $Z^{k_+}_\beta$ which contain the orbit of $x$, namely such that
\begin{equation}\label{eq:towerspm}
T^{-j-1} (x) \in Z^{k_+}_{\beta^-}, \qquad T^{q^{k_+}_\beta-j} (x) \in Z^{k_+}_{\beta^+}.
\end{equation}
Since by the BC condition $\ZC{n_k}{n_{k+1}}$ is a positive matrix, each tower $Z^{k_+}_{\beta}$ of level $k+1$ is obtained stacking at least once each the previous level $Z^{k}_{\alpha}$, $\alpha\in \mathcal{A}$. In particular, there exists a floor $F_-$ of $Z^{k_+}_{\beta^-}$ and a floor $F_+$ of $Z^{k_+}_{\beta^+}$ which belong to $I^k_{\alpha}$. Since the (full) orbit of $x$ under $T$ visits both $Z^{k_+}_{\beta_-}$ and $Z^{k_+}_{\beta_+}$ by definition of $\beta_\pm$, see \eqref{eq:towerspm}, it visits every floor of both, so there exists $j_-,j_+\geq 0$ such that
\begin{equation}\label{eq:jchoice}
T^{-j_-}(x) \in F_-\subseteq I^{k}_{\alpha}\cap Z^{k_+}_{\beta^-}, \qquad T^{j_+}(x) \in F_+ \subseteq I^{k}_{\alpha} \cap Z^{k_+}_{\beta^+}.
\end{equation}

We can now define $m_{-k}$ and $m_k$ as
\begin{equation}\label{eq:nkdef}
m_{-k}:=-j_-+i, \qquad m_{k}:=j_+ + i.
\end{equation} 
Notice that, by construction, in view of \eqref{eq:jchoice} and \eqref{eq:nkdef}, $T^{m_{-k}}(x)$ and $T^{m_k}(x)$ both belong, as $x$, to the $i$-th floor of the tower $Z^{k}_{\alpha}$. 

Let us illustrate the notations in the following picture.

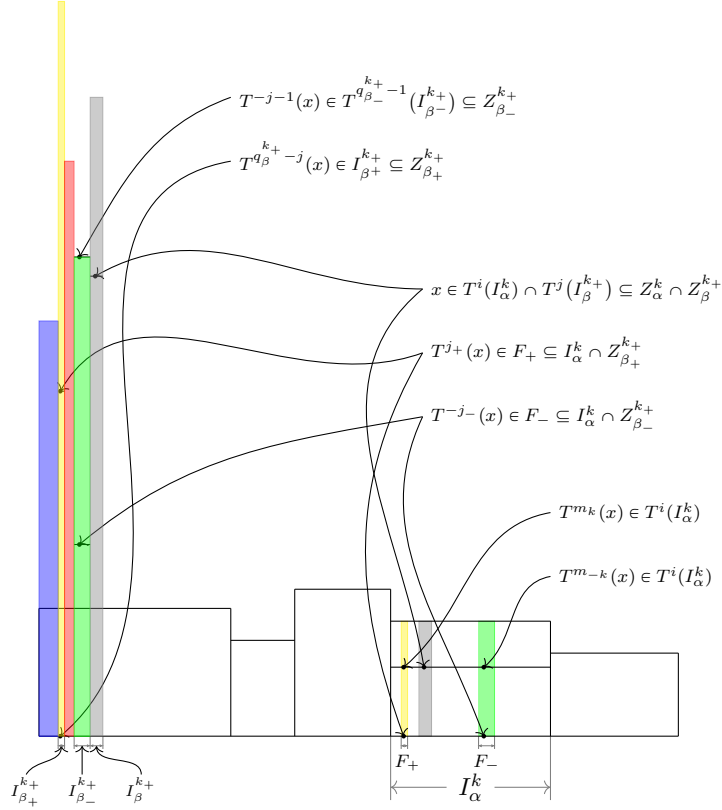
\begin{figure}[h!]
\centering
\begin{tikzpicture}[scale=0.85, transform shape]
\def\a{2}
\def\b{1.5}
\def\c{2.3}
\def\d{1.8}
\def\e{1.3}
\def\A{3}
\def\B{4}
\def\C{5.5}
\def\D{8}
\def\E{10}

\draw [yellow, fill=yellow, opacity=0.4] (\C * 1.03,0) rectangle (\C*1.03 + \B*0.1 - \A*0.1,\d);
\draw [gray, fill=gray, opacity=0.4] (\C * 1.08,0) rectangle (\C*1.08 + \E*0.1 - \D*0.1,\d);
\draw [green, fill=green, opacity=0.4] (\C * 1.25,0) rectangle (\C*1.25 + \D*0.1 - \C*0.1,\d);

\node[below] at (-0.2,-0.5) {\textcolor{black}{\tiny$I^{k_+}_{\beta_+}$}};
\draw[<-,thin] (\A*0.05 + \B * 0.05, -0.15) to [out=270,in=60] (-0.2,-0.6) {};
\draw[-, gray] (\A*0.1,-0.2) -- (\A*0.1,0)  {};
\draw[-, gray] (\B*0.1,-0.2) -- (\B*0.1,0)  {};
\draw[{Latex[scale=0.3]}-{Latex[scale=0.3]},line width=0.05mm, gray]  (\A*0.1,-0.15)--(\B*0.1,-0.15);

\node[ below] at (0.7,-0.5) {\textcolor{black}{\tiny$I^{k_+}_{\beta_-}$}};
\draw[<-,thin] (\C*0.05 + \D * 0.05, -0.15) to [out=270,in=90] (\C*0.05 + \D * 0.05, -0.6)  {};
\draw[-, gray] (\C*0.1,-0.2) -- (\C*0.1,0)  {};
\draw[{Latex[scale=0.3]}-{Latex[scale=0.3]},line width=0.05mm, gray]  (\C*0.1,-0.15)--(\D*0.1,-0.15);

\node[ below] at (1.6,-0.5) {\textcolor{black}{\tiny$I^{k_+}_\beta$}};
\draw[<-,thin] (\D*0.05 + \E * 0.05, -0.15) to [out=270,in=90] (1.6,-0.6)  {};
\draw[-, gray] (\D*0.1,-0.2) -- (\D*0.1,0)  {};
\draw[-, gray] (\E*0.1,-0.2) -- (\E*0.1,0)  {};
\draw[{Latex[scale=0.3]}-{Latex[scale=0.3]},line width=0.05mm, gray]  (\D*0.1,-0.15)--(\E*0.1,-0.15);

\node[] at (\C*0.5 + \D*0.5,-0.8) {\textcolor{black}{$I^k_\alpha$}};
\draw[-,gray] (\C, -1) -- (\C, 0);
\draw[-,gray] (\D, -1) -- (\D, 0);
\draw[<-,gray] (\C, -0.8) -- (\C + \D*0.35 - \C*0.35, -0.8);
\draw[<-,gray] (\D, -0.8) -- (\D + \C*0.35 - \D*0.35, -0.8);
\draw[-, gray] (\C*1.03,-0.2) -- (\C*1.03,0)  {};
\draw[-, gray] (\C*1.03 + \B*0.1 - \A*0.1,-0.2) -- (\C*1.03 + \B*0.1 - \A*0.1,0)  {};
\draw[{Latex[scale=0.3]}-{Latex[scale=0.3]},line width=0.05mm, gray]  (\C*1.03,-0.15)--(\C*1.03 + \B*0.1 - \A*0.1,-0.15);
\node[below] at (\C*1.03 + \B*0.05 - \A*0.05 +0.05,-0.15) {\textcolor{black}{\scriptsize$F_+$}};
\draw[-, gray] (\C * 1.25,-0.2) -- (\C * 1.25,0)  {};
\draw[-, gray] (\C*1.25 + \D*0.1 - \C*0.1,-0.2) -- (\C*1.25 + \D*0.1 - \C*0.1,0)  {};
\draw[{Latex[scale=0.3]}-{Latex[scale=0.3]},line width=0.05mm, gray]  (\C*1.25,-0.15)--(\C*1.25 + \D*0.1 - \C*0.1,-0.15);
\node[below] at (\C*1.25 + \D*0.05 - \C*0.05,-0.15) {\textcolor{black}{\scriptsize$F_-$}};


\node[circle, fill, scale=0.2, label=below:] at (\A * 0.1 + \C * 0.006, 0) {}; 
\node[circle, fill, scale=0.2, label=below:] at (\C * 0.115 , \b*5) {}; 
\node[circle, fill, scale=0.2, label=below:] at (\D * 0.11, \d*4) {}; 

\node[circle, fill, scale=0.2, label=below:] at (\C * 1.095, \d*0.6) {}; 
\node[circle, fill, scale=0.2, label=below:] at (\C * 1.037, 0) {}; 
\node[circle, fill, scale=0.2, label=below:] at (\C * 1.265, 0) {}; 


\draw[<-,thin] (\A * 0.1 + \C * 0.006, 0) to [out=40,in=190] (3,9) {};
\node[ right] at (3,9) {\textcolor{black}{\scriptsize$T^{q^{k_+}_\beta-j}(x) \in I_{\beta^+}^{k_+} \subseteq Z_{\beta_+}^{k_+}$}};

\draw[<-,thin] (\C * 0.115 , \b*5) to [out=60,in=200] (3,10) {};
\node[ right] at (3,10) {\textcolor{black}{\scriptsize$T^{-j -1}(x) \in T^{q^{k_+}_{\beta_-}-1}\big(I_{\beta^-}^{k_+}\big) \subseteq Z_{\beta_-}^{k_+}$}};

\draw[<-,thin] (\D * 0.11, \d*4) to [out=30,in=180] (6,7) {};

\draw[<-,thin] (\C * 1.037, 0) to [out=110,in=240] (6,6) {};
\node[ right] at (6,6) {\textcolor{black}{\scriptsize$T^{j_+}(x) \in F_+ \subseteq I^k_\alpha \cap Z_{\beta_+}^{k_+}$}};

\draw[<-,thin] (\C * 1.095, \d*0.6) to [out=100,in=220] (6,7) {};
\node[ right] at (6,7) {\textcolor{black}{\scriptsize$x \in T^i(I_\alpha^k) \cap T^j\big(I_\beta^{k_+}\big) \subseteq Z^k_\alpha \cap Z_{\beta}^{k_+}$}};

\draw[<-,thin] (\C * 1.265, 0) to [out=110,in=245] (6,5) {};
\node[ right] at (6,5) {\textcolor{black}{\scriptsize$T^{-j_-}(x) \in F_- \subseteq I^k_\alpha \cap Z_{\beta_-}^{k_+}$}};


\draw[<-,thin] (\A * 0.1 + \C * 0.006, \d*3) to [out=60,in=200] (6,6) {};
\node[circle, fill, scale=0.2, label=below:] at (\A * 0.1 + \C * 0.006, \d*3) {};
\draw[-,thin] (\A*0.1, \d*3) -- (\B*0.1, \d*3) {};

\draw[<-,thin] (\C * 0.115, \b*2) to [out=40,in=190] (6,5) {};
\node[circle, fill, scale=0.2, label=below:] at (\C * 0.115, \b*2) {};
\draw[-,thin] (\C*0.1, \b*2) -- (\D*0.1, \b*2) {};

\draw[-,thin] (\D*0.1, \d*4) -- (\E*0.1, \d*4) {}; 
\draw[-,thin] (\C*0.1, \b*5) -- (\D*0.1, \b*5) {};

\draw[-,thin] (\C, \d*0.6) -- (\D, \d*0.6) {}; 
\node[circle, fill, scale=0.2, label=below:] at (\C * 1.037, \d*0.6) {};
\draw[<-,thin] (\C * 1.037, \d*0.6) to [out=40,in=190] (8,3.5) {};
\node[ right] at (8,3.5) {\textcolor{black}{\scriptsize$T^{m_k}(x) \in T^i(I_\alpha^k)$}};

\node[circle, fill, scale=0.2, label=below:] at (\C * 1.265, \d*0.6) {};
\draw[<-,thin] (\C * 1.265, \d*0.6) to [out=40,in=190] (8,2.5) {};
\node[ right] at (8,2.5) {\textcolor{black}{\scriptsize$T^{m_{-k}}(x) \in T^i(I_\alpha^k)$}};

\draw[-] (0,0) -- (10,0) node[right] {}; 
\draw[-] (0,0) -- (0,\a) node[right] {}; 
\draw[-] (\A,0) -- (\A,\a) node[right] {}; 
\draw[-] (\B,0) -- (\B,\c) node[right] {}; 
\draw[-] (\C, 0) -- (\C,\c) node[right] {}; 
\draw[-] (\D, 0) -- (\D,\d) node[right] {}; 
\draw[-] (\E, 0) -- (\E,\e) node[right] {}; 

\foreach \height/\base/\baseend in {\a/0/\A,\b/\A/\B,\c/\B/\C,\d/\C/\D,\e/\D/\E}{
\draw[-] (\base, \height) -- (\baseend,\height) node[right] {}; }

\foreach \col/\height/\base/\baseend in {blue/\e/0/\A,yellow/\c/\A/\B,red/\d/\B/\C,green/\b/\C/\D,gray/\a/\D/\E}{\draw [\col, fill=\col, opacity=0.4] (\base*0.1,0) rectangle (\baseend*0.1,\height*5); }
\end{tikzpicture}
\caption{\small The Rohlin towers at the $n_k$-th and $(n_k + N)$-th steps of induction are represented by white and colored rectangles, respectively.}
\end{figure}

Let us show that with this definition, the sequence $(m_k)_{k \in \Z}$ satisfies \eqref{eq:boundedBS}. We will only prove that
\[ \sup_{k\in \N} |S_{m_k} f_\omega(x)| < +\infty,\] 
since the proof for the sequence $m_{-k}$ is completely analogous. Indeed, it suffices to consider the backward orbit of $x$ instead of the forward one. 

We start by showing that we can decompose the Birkhoff sum $S_{m_k} f_\omega (x)$ into a sum of special Birkhoff sums of level $n_k$ plus an initial and final segment. More precisely, we will show that we can express $S_{m_k} f_\omega (x)$ as 
\begin{equation}\label{eq:BSdecompnk}
\begin{split}
S_{m_k} f_\omega (x ) = S_{q^k_\alpha -i}\, f_\omega (x) + & \sum_{0\leq \ell \leq \ell_0} \SBS{n_k}{f_\omega} (T_k^\ell (y)) + S_{i} f_\omega (z), 
\end{split}
\end{equation}
for some $\ell_0 \in \N$, where $y:= T^{q^k_\alpha -i}(x)$ and $ z:= T_k^{m_k - i} (x)$.

Let us prove \eqref{eq:BSdecompnk}. By definition of $y$ and $z$, it follows that
\begin{equation}
\label{eq:decomp}
\begin{aligned}
 S_{m_k} f_\omega (x ) & = S_{q^k_\alpha - i}\, f_\omega (x) + S_{m_k-q^k_\alpha} f_\omega (y) + S_i f_\omega (z). 
\end{aligned}
\end{equation}
Notice that, since $x$ belongs to the $i$-th floor of the tower $Z^{k}_{\alpha}$, we have $y \in I^k$. Consider now the successive iterates $T_k^\ell (y)$, $\ell\in \mathbb{N}$, which by definition of the induced map $T_k$ all belong to $I^k$, and let $\ell_0$ be the last one (in the natural order of the orbit of $x$ under $T$) before $T^{m_k}(x)$. Notice that $T_k^{\ell_0}(y) = z = T^{m_k - h_\alpha^k}(y)$. For $0\leq \ell \leq \ell_0$, let $\alpha_\ell \in \A$ be such that $T_k^\ell(y) \in I^k_{\alpha_\ell}$. 

Since, for any $0 \leq \ell \leq \ell_0$, we have
\[ T_k(T_k^\ell(y)) = T^{h_{\alpha_{\ell}}^k}(T_k^\ell(y)), \qquad S_{h_{\alpha_{\ell}}^k}{f_\omega}(T_k^\ell(y)) = \SBS{n_k}{f_\omega}(T_k^\ell(y)),\]
it follows from the cocycle relation \eqref{eq:cocyclerel}, that 

\begin{equation}\label{eq:ellbound}
S_{m_k-q^k_\alpha} f_\omega (y) = \sum_{0\leq \ell \leq \ell_0} S_{h_{\alpha_{\ell}}^k}{f_\omega}(T_k^\ell(y)) = \sum_{0\leq \ell \leq \ell_0} \SBS{n_k}{f_\omega}(T_k^\ell(y)),
\end{equation}
which together with \eqref{eq:decomp} conclude the proof of \eqref{eq:BSdecompnk}. 

Notice now that, since $\ell_0$ is bounded by the number of towers of level $n_k$ contained in the union of $Z^{k_+}_{\beta}$ and $Z^{k_+}_{\beta^+}$, by the interpretation of the entries of the incidence matrices (see \S~\ref{sc:incidence_matrices}) and the BC condition, it follows that 
\begin{equation}\label{eq:l0bound}
\ell_0+1 \leq 2 \Vert \ZC{n_k}{n_k + N}\Vert \leq 2 K.
\end{equation}
Moreover, observe that the initial and final sum in \eqref{eq:BSdecompnk} combine into a full special Birkhoff sum since they combine to a sum over exactly one point for each floor of the tower $Z_\alpha^k$, and $f_\omega$ is constant on each floor of $Z_\alpha^k$, so that
\begin{align*}
S_{q^k_\alpha -i}\, f_\omega (x) + S_i f_\omega (z) & = \sum_{m=0}^{q^k_\alpha-i+1} f_\omega (T^m x) + \sum_{m=q^k_\alpha-i}^{q^k_\alpha-1} f_\omega (T^m (T^{-i}z) \\ & = \sum_{n=0}^{q^k_\alpha-1} f_\omega (T^m x) = \SBS{n_k}{f_\omega}(x)= \omega^{(n_k)}_{\alpha}.
\end{align*}
Using these last two observations, we can estimate \eqref{eq:BSdecompnk} with a bounded number of special Birkhoff sums, which in turn, in view of \eqref{eq:log_slopes_height_cocycle}, \eqref{SBS_cocycle_relation}, \eqref{eq:ellbound}, \eqref{eq:l0bound} and the BC condition, yields
\begin{align*}
\| S_{m_k} f_\omega (x ) \| & \leq \left| \omega^{(n_k)}_{\alpha}\right|+ \left| \sum_{0\leq \ell \leq \ell_0} \SBS{n_k}{f_\omega} (T_k^\ell (y)) \right| \\ & = \left| \omega^{(n_k)}_{\alpha}\right| + \left| \sum_{0\leq \ell \leq \ell_0} \omega^{(n_k)}_{\alpha_\ell} \right| \\ 
& \leq (\ell_0+2) \Vert \omega^{(n_k)} \Vert\\
& \leq (2K+1)V\|\omega\|.
\end{align*}
Therefore $ \sup_{k\in \N} |S_{m_k} f_\omega(x)| < +\infty,$ which concludes the proof.
\end{proof}

Let us also show how Corollary~\ref{cor:boundedseq} follows from Proposition~\ref{prop:boundedseq}.

\begin{proof}[Proof of Corollary~\ref{cor:boundedseq}]
Notice that since $T_0$ satisfies the BC condition, in particular, its rotation number $\gamma(T_0)$ is  $\infty$-complete. Furthermore, 
since $\omega\in E_{cs}(T_0)$ and $E_{cs}(T_0) \subset \lambda^\bot$ (see \eqref{rk:centralinclusion}), the assumptions of  Proposition~\ref{prop:non_empty_affine} hold and therefore  there exists an AIET $T$ in $ \textup{Aff}(\gamma(T_0),\omega)$. By 
Proposition \ref{prop:semi-conjugacy}, there also exists a semiconjugacy between $T$ and $T_0$ i.e. an increasing surjective map $h: [0, 1) \to [0, 1)$, satisfying $T_0 \circ h = h \circ T$.  
Notice that $h$ maps continuity intervals of $T$ onto continuity intervals of $T_0$. Therefore, by definition of the functions $f_{T,\omega}$ and $f_{T_0,\omega}$ (see Definition~\ref{def:fv}), 
$f_{T_0, \omega}\circ h= f_{T,\omega}$. Thus, for every $n\in \mathbb{N}$ and any $x\in [0,1]$, 
$$
S_n^{T_0} f_{T_0,\omega} (h(x)) = \sum_{k=0}^{n-1} f_{T_0,\omega} ((T_0)^k\circ h (x))=  \sum_{k=0}^{n-1} f_{T_0,\omega}( h \circ T^k(x)) = S_n^{T} f_{T,\omega} (x) .
$$
Similarly (recalling Definition~\ref{def:BS}), one sees that  $S_n^{T_0} f_{T_0,\omega} \circ h(x) = S_n^{T} f_{T,\omega} (x)$ for any $n\in \mathbb{Z}$.  Thus, the conclusion of   the corollary follows immediately from Proposition~\ref{prop:boundedseq}.
\end{proof}
\section{Singularity of the invariant measure}
\label{sc:singularity}

In this section, we present a direct proof of Theorem \ref{thm:regularity}. We first state a standard analysis Lemma (Lemma~\ref{lem:locally_constant} below), which will be helpful in the proof and roughly says that an integrable function is locally constant when looking at it on a sufficiently small scale. 

Given an integrable function $\psi: [0, 1] \to \R$, for any $x \in [0,1]$ and any $r, \delta > 0$, we denote
\[ E^\psi_r(x, \delta) = \{ y \in [0, 1] \mid |x - y| < r;\, |\psi(x) - \psi(y)| > \delta\}.\]
The set $E^\psi_r(x, \delta)$ is the set of `exceptional points' in the ball of radius $r$ around $x$ for which the value of $\psi$ differs from $\psi(x)$ by more than $\delta$. 

\begin{lemma}
\label{lem:locally_constant}
Let $\psi: [0, 1] \to \R$ be an integrable function and let $\delta > 0$. Given $\epsilon > 0$, there exists $r_0 > 0$ such that
\begin{equation}
\label{eq:measure_good_points}
\textup{Leb}\left(\left\{ x \in [0, 1] \,\left|\, \big| E^\psi_r(x, \delta) \big| < \tfrac{2r\epsilon}{\delta}\, \text{ for all } 0 < r < r_0\right. \right\}\right) > 1 - \epsilon.
\end{equation}
\end{lemma}

This lemma is a simple consequence of the Lebesgue differentiation theorem and Egorov's theorem. 

\begin{proof}[Proof of Theorem \ref{thm:regularity}]
Let $T_0$ and $T$ as in Theorem \ref{thm:topconjugacy}, and assume WLOG that $T_0$ is uniquely ergodic. Furthermore, assume that $T_0$ verifies the HS Condition (see Definition \ref{def:HS_Condition}). Denote by $h$ the associated conjugating map $h$ verifying $T \circ h = h \circ T_0$. Recall that by the discussion at the beginning of \S\ref{sc:singularity}, this measure is either singular or absolutely continuous with respect to the Lebesgue measure. 

Suppose by contradiction that the unique invariant probability measure $\mu$ of $T$ is absolutely continuous with respect to Lebesgue and denote by $\varphi$ the associated Radon-Nykodim derivative, namely, $\mu = \varphi \textup{Leb}$. Since $\mu$ is $T$-invariant, 
\begin{equation}
\label{eq:invariant_density}
(\varphi \circ T) T' = \varphi
\end{equation} 
almost surely. Recall that by Corollary \ref{cor:central_lower_bound}, 
\[ \delta := \tfrac{1}{4}\inf_{n \in \N} |\omega^{(m_k)}| > 0,\]
where the sequence $m_k$ denotes the sequence of Zorich times associated with $T_0$. We will derive a contradiction by showing that \eqref{eq:invariant_density} implies that the norm of $\omega^{(m_k)}$ may be arbitrarily small. 

\noindent By iterating \eqref{eq:invariant_density} and taking logarithm, it follows that, for any $k \in \N$,
 $$(\log \varphi \circ T^k) - \log \varphi = S_k^T f_\omega^T $$ 
almost surely, where $\omega$ denotes the log-slope vector of $T$. Composing the previous equality with $h$ and denoting $\psi = \log \varphi \circ h$, we obtain,
 \[ \psi \circ T_0^k - \psi = S_k^{T_0}f_\omega^{T_0} \]
almost surely. Considering the return times given by the Rauzy-Veech induction yields
\[ \psi \circ T_0^{q^{(n)}_\alpha}(x) - \psi(x) = \omega^{(n)}_\alpha, \]
for $x \in I^{(n)}_\alpha$ and $\alpha \in \A.$ Let $(n_k)_{k \in \N}$ be the subsequence of Zorich times given by Proposition \ref{prop:fullmeasure}, and let us denote
\[ q^k := q^{(n_k)}, \quad \omega^k := \omega^{(n_k)}, \quad I^k = I^{(n_k)}_\alpha.\]
Notice that by Condition HS, 
\begin{equation}
\label{eq:extended_SBS}
\psi \circ T_0^{q^k_\alpha}(x) - \psi(x) = \omega_\alpha,
\end{equation}
for $x \in \bigcup_{i = 0}^{h_k} T^i(I^{(n)}_\alpha)$ and $\alpha \in \A$, 
where $h_k = \tfrac{1}{4} \min_{\alpha \in \A} q^k_\alpha.$ 
Since $T_0$ verifies the BC Condition, there exists $0 < c_0 < 1$ such that
\begin{equation}
\label{eq:balanced_acc_lengths}
\min_{\alpha \in \A} \frac{|I^k_\alpha|}{|I^k|} > c_0,
\end{equation}
for any $k \in \N$. Denote $\epsilon = \tfrac{c_0}{8C}\min\{1, \delta\}$, where $C$ is given by the HS Condition, and let $r_0$ be given by Lemma \ref{lem:locally_constant}. Let $k$ sufficiently large so that $|I^k| < r_0$. Since 
\[ \textup{Leb}\left( \bigcup_{\alpha \in \A} \bigcup_{i = 0}^{h_k} T^i(I^k) \right) > \tfrac{1}{4C},\]
it follows from \eqref{eq:measure_good_points} that there exists $x_k = T^{i_k}(r_k)$, with $r_k \in \big[\tfrac{|I^k|}{2}, |I^k|\big)$ and $0 \leq i_k \leq h_k$, such that $$|E^\psi_{r_k}(x_k, \delta)| < \tfrac{2r_k\epsilon}{\delta}.$$
Fix $\alpha \in \A$. Notice that 
\[T^{i_k}(I_\alpha^k), T^{q^k_\alpha + i_k}(I_\alpha^k) \subset B_{r_k}(x_k),\]
and by \eqref{eq:balanced_acc_lengths},
\[|T^{i_k}(I_\alpha^k)|, |T^{q^k_\alpha + i_k}(I_\alpha^k)| \geq c_0 |I^k|.\]
Hence
\begin{align*}
\textup{Leb}& \left(\left\{ y \in T^{i_k}(I_\alpha^k) \,\left|\, y \notin E^\psi_{r_k}(x_k, \delta); \, T^{q_\alpha^k}(y) \notin E^\psi_{r_k}(x_k, \delta) \right. \right\} \right) \\
& \geq |I^k_\alpha| - 2 |E^\psi_{r_k}(x_k, \delta)| \geq |I^k_\alpha| - \frac{4r_k\epsilon}{\delta} \\
& \geq |I^k_\alpha| \left( 1 - \frac{4\epsilon}{\delta c_0}\right) \geq \frac{1}{2} |I^k_\alpha|.
\end{align*}
Thus, there exists $y^k_\alpha \in T^{i_k}(I_\alpha^k)$, verifying \eqref{eq:extended_SBS}, such that 
\[y^k_\alpha, T^{q^k_\alpha} (y^k_\alpha) \notin E^\psi_{r_k}(x_k, \delta).\]
Hence $$|\psi(y^k_\alpha) - \psi(T^{q^k_\alpha} (y^k_\alpha))| < 2\delta,$$ and by \eqref{eq:extended_SBS}, $|\omega^k_\alpha| < 2\delta$. 
Since $\alpha \in \A$ was arbitrary, and by definition of $\delta$, we have $|\omega^k| < 2\delta,$ we reached a contradiction. This concludes the proof.
\end{proof}

\section{Full measure of the IETs conditions}\label{sc:fullmeasureproof}
This section is devoted to the proof of Proposition \ref{prop:fullmeasure}, namely show simultaneously that the BC and the HS Conditions introduced in \S~\ref{sec:fullmeasure} (see Definitions~\ref{def:BC_condition} and~\ref{def:HS_Condition}) are satisfied by a full measure set of (irreducible) IETs. We start by introducing a few objects and notations needed in the proof. 

\subsection{Oseledet's splittings}
We denote the \emph{natural extensions} of the Zorich map $\ZorichMap$ and of the Zorich renormalization $\ZorichNorm$ by 
\[\ZorichMapExt: \DomExt \to \DomExt, \quad\quad \ZorichNormExt: \DomExtNorm \to \DomExtNorm.\]
The domains of these transformations admit a geometric interpretation in terms of \emph{zippered rectangles}, which were introduced by W. Veech \cite{veech_gauss_1982} when considering suspensions over IETs, and can be seen as subsets of $\Dom \times \R^\A$. We denote points in the domains $\DomExt$ and $\DomExtNorm$ by $\IETExt.$ 

Recall that $\ZorichNorm$ admits an unique invariant probability measure $\mu_{\ZorichNorm}$ equivalent to the Lebesgue measure and its natural extension $\ZorichNormExt$ admits an unique invariant probability measure $\mu_{\ZorichNormExt}$ equivalent to Lebesgue and such that $p_*(\mu_{\ZorichNormExt}) = \mu_{\ZorichNorm}$, where $p: \DomExtNorm \to \DomExt$ denotes the canonical projection $p\IETExt = \IET.$

Considering the natural extension of the Zorich renormalization and extending the cocycle $B$ trivially to $\DomExtNorm$ using the canonical projection $p: \DomExtNorm \to \DomExt$, the height and length cocycles admit invariant Oseledet's splittings
\begin{gather*}
 E_s\IETExt \oplus E_c\IETExt\oplus E_u\IETExt =\R^{\A},\\
 F_s\IETExt \oplus F_c\IETExt\oplus F_u\IETExt =\R^{\A},
\end{gather*}
respectively, corresponding to the sets of vectors with negative, zero, and positive Lyapunov exponents. These spaces verify
 \begin{equation}
 \label{eq:relation_flag_splitting}
 \begin{aligned}
& E_s\IETExt = E_s \IET, \quad\quad E_c\IETExt \oplus E_s\IETExt = E_{cs}\IET, \\
& F_s\IETExt = F_s \IET, \quad\quad F_c\IETExt \oplus F_s\IETExt = F_{cs}\IET,
 \end{aligned}
 \end{equation}
for a.e. $\IETExt \in \DomExtNorm$. Moreover, since the height and length cocycles are \emph{dual} to each other, we have
\begin{equation}
\label{eq: orthogonal_flags}
E_s\IETExt = F_{cs}\IETExt ^\bot, \hskip1cm F_s\IETExt = E_{cs}\IETExt ^\bot,
\end{equation}
for a.e. $\IETExt \in \DomExt.$ We refer the interested reader to \cite{zorich_deviation_1997} for a precise definition of dual cocycle. 

For the sake of simplicity, for a.e. $\IETExt \in \DomExt$ and for any $n \in \Z$, we denote their iterates under $\ZorichMapExt$ by $\IETExtn{n}$ and the associated Oseledets subspaces by $E^n_\epsilon\IETExt = E_\epsilon\IETExtn{n}$, where $\epsilon \in \{s, c, u\}$. 

\subsection{Angle control between the splittings}

Recall that the angle between two subspaces $\{0\} \subsetneq E, F \subsetneq \R^\A$ is given by 
\[\angle(E, F) = \min \left\{ \arccos \left( \left | \langle v, w \rangle \right| \right) \mid v \in E, w \in F, |v| = 1 = |w|\right\}.\]
Denoting by $\pi_{E, F}: E \to F$ the projection of $E$ to $F$, we have 
\[\| \pi_{E, F} \| \leq \cos\angle(E, F).\]
This implies the following.
\begin{lemma}
\label{lem: bound_projection}
Let $\{0\} \subsetneq E, F \subset \R^\A$ and $\delta = \cos\angle(E, F)$. Then 
\[ \sqrt{1 - \delta} |v| \leq |\pi_{F^\bot}(v)|,\]
for any $v \in E$, where {$\pi_{F^\bot}$} denotes the orthogonal projection to $F^\bot$. 
\end{lemma}

The following observation will be of fundamental importance.  For a proof see \cite[Proposition 7.6]{yoccoz_interval_2010}.

\begin{proposition}
\label{prop: trivial_action}
For a.e. $\IETExt \in \DomExt$ and for any $n \in \Z$,
\[ \LC{0}{n}\IETExt (\Kpi) = \Kpim{n}. \]
Moreover, it is possible to pick a base of $\Kpi$, for each $\pi \in \PermSpace$, such that, for a.e. $\IETExt \in \DomExt$ and for any $n \in \Z$, the matrix associated to the transformation 
\[ \LC{0}{n}\IETExt \mid_{\Kpi}: \Kpi \to \Kpim{n},\]
 with respect to the selected basis, is the identity. 
\end{proposition}

The previous proposition shows that the central space for the length cocycle is given by 
\begin{equation}
\label{eq: kernel_central}
F_c\IETExt = \textup{Ker}(\Omega_\pi),
\end{equation}
for a.e. $\IETExt \in \DomExtNorm$. Applying Proposition \ref{prop: trivial_action}, we can show the following. 

\begin{lemma}
\label{lem:bounds_central}
There exists $C_0 > 1$ such that for a.e. $\IET$, 
\[ \left \| \pi_{\Kpim{n}} \circ \HC{0}{n} \right\| \leq C,\]
for any $n \in \N$. Moreover, if $\pi^n = \pi$ for some $n \in \N$, then 
\[ \pi_{\Kpim{n}} \circ \HC{0}{n} = \pi_{\Kpim{n}}. \]
\end{lemma}
\begin{proof}
Let $n \in \N$ be fixed. Notice that $\big(\HC{0}{n}\big)^{-1} = \big(\LC{0}{n}\big)^T$. Hence, for any $v \in \R^\A$ and any $w \in \Kpim{n}$,
\begin{align*}
 \big \langle \HC{0}{n} v, w \big\rangle & = \big\langle v, \big(\HC{0}{n}\big)^T w \big\rangle = \big\langle v, \big(\LC{0}{n}\big)^{-1} w \big\rangle .
\end{align*}
By Proposition \ref{prop: trivial_action}, it follows that 
\[ \big| \big \langle \HC{0}{n} v, w \big\rangle \big| \leq C_0|v||w| \]
for some constant $C_0$ depending only on $d$, and if $\pi^n = \pi$,
\[ \big \langle \HC{0}{n} v, w \big\rangle = \langle v, w\rangle,\]
which proves the lemma. 
\end{proof}

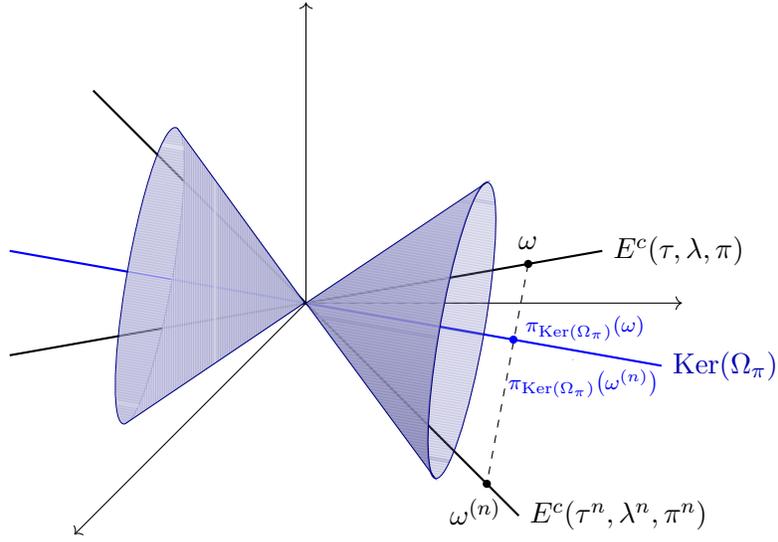
\begin{figure}
\centering
\begin{tikzpicture}
  \def\R{2.8}
  \coordinate (O) at (0,0);
  \coordinate (J1) at (170:\R); 
  \coordinate (-J1) at (170:-\R);
  \coordinate (J2) at (-10:\R);
  \draw[->] (0,0,0) -- (5,0,0) node[right] {}; 
  \draw[->] (0,0,0) -- (0,4,0) node[right] {}; 
  \draw[->] (0,0,0) -- (0,0,8) node[above] {}; 
\draw[thick,blue] (170:4) -- (170:-4.8);
\draw[thick,black] (190:4) -- (190:-4);
\draw[thick,black] (135:4) -- (135:-4);
\draw[thin,dashed,black] (190:-3) -- (135:-3.4);

\node[circle, fill, scale=0.3, label=above:$\omega$] at (190:-3) {};
\node[circle, fill=blue, scale=0.3,label=above:] at (170:-2.8) {};
\node[circle, fill=blue, scale=0.0,label={[text=blue]:{\scriptsize$\quad\pi_{\textup{Ker}(\Omega_\pi)}(\omega)$}}] at (170:-3.62) {};
\node[vector,right] at (170:-4.8) {$\textup{Ker}(\Omega_\pi)$};
\node[circle, fill, scale=0.3, label=below:$\omega^{(n)}\hspace{0.3cm}$] at (135:-3.4) {};
\node[circle, fill=blue, scale=0.0,label={[below,text=blue]:{\scriptsize$\quad\pi_{\textup{Ker}(\Omega_\pi)}\big(\omega^{(n)}\big)$}}] at (168:-3.62) {};
\node[vector,black,right] at (190:-4) {$E^c(\tau, \lambda, \pi)$};
\node[vector,black,right] at (135:-4) {$E^c(\tau^n, \lambda^n, \pi^n)$};
\jetcone{O}{J1}{2}{0.3}
\jetconee{O}{J2}{2}{0.3}
\end{tikzpicture}
\caption{\small Denoting $\omega^{(n)} = \HC{0}{n}\IETExt\omega$, for every $n$ such that $\pi^{n} = \pi$ the projection of $\omega^{(n)}$ to $\textup{Ker}(\Omega_\pi)$ remains constant (see Lemma \ref{lem:bounds_central}). Moreover, whenever the angle between $E^c(\tau^n, \lambda^n, \pi^n)$ and $\textup{Ker}(\Omega_\pi)$ is small, the norm of $\omega^{(n)}$ is bounded, from above and below, by positive constants depending only on $\omega$ (see Corollary \ref{cor:central_lower_bound} and Lemma \ref{lem: bound_cocycle_central}).}
\end{figure}

\begin{corollary}
\label{cor:central_lower_bound}
For a.e. $\IETExt$ and for any $v \in E_c\IETExt \setminus \{0\}$ 
\[ \inf_{n \in \N} |\HC{0}{n}v| > 0. \] 
\end{corollary}
\begin{proof}
By \eqref{eq: orthogonal_flags} \eqref{eq: kernel_central}, 
\begin{align*}
E_c(\tau, \lambda, \pi) \cap {\Kpi}^\bot & \subseteq E_c(\tau, \lambda, \pi) \cap {\Kpi}^\bot  \cap E_{cs}(\tau, \lambda, \pi) \\
& =  E_c(\tau, \lambda, \pi) \cap {\Kpi}^\bot  \cap F_s(\tau, \lambda, \pi)^\bot \\
& = E_c(\tau, \lambda, \pi) \cap ({\Kpi} \oplus F_s(\tau, \lambda, \pi))^\bot  \\
& = E_c(\tau, \lambda, \pi) \cap F_{cs}(\tau, \lambda, \pi)^\bot \\ 
& = E_c(\tau, \lambda, \pi) \cap E_s(\tau, \lambda, \pi),
\end{align*}
for a.e. $\IETExt.$ Hence, 
\begin{equation}
\label{eq:iso_central_kernel}
E_c\IETExt \cap \orth{\Kpi} = \{0\},
\end{equation}
for a.e. $\IETExt,$ since, otherwise, the stable and central spaces associated with the height cocycle would have a non-trivial intersection. By Lemma \ref{lem:bounds_central}, the projection of $\HC{0}{n}v$ to the spaces $\Kpi$ is constant. Therefore, its norm is uniformly bounded from below. 
\end{proof}

\begin{lemma}
\label{lem: bound_cocycle_central}
For a.e. $\IETExt \in \DomExtNorm$ and for any $n \geq m \geq 0$, there exists a constant $C \IETExtn{n} > 0$, depending only on $\angle(E_c^n\IETExt , \orth{\Kpim{n}})$, such that the linear \mbox{operator} 
\[\HC{m}{n}\IETExt \mid_{E^m_c\IET}: E^m_c\IETExt \to E^n_c\IETExt\]
 verifies 
\[ \left \|\HC{m}{n}\IETExt \mid_{E^m_c\IETExt } \right \| \leq C\IETExtn{n}.\]
\end{lemma}
\begin{proof}
Let $\IETExt \in \DomExtNorm$ be Oseledets generic and fix $n > m \geq 0$. By \eqref{eq:iso_central_kernel}, we may assume without loss of generality that
\[E_c^n\IETExt \cap \Kpim{n}^\bot = \{0\}, \]
 for any $n \in \N$. Thus, by  Lemma \ref{lem: bound_projection}, {applied to $E:=E^m_c\IETExt$ and $F:={\Kpim{n}}$,}
\begin{equation}
\label{eq:bound_projection}
\left \|\HC{m}{n} \mid_{E^m_c\IETExt } \right \| \leq C\IETExtn{n} \left \| \pi_{\Kpim{n}} \circ \HC{n}{m} \mid_{E^m_c\IETExt } \right \|,
\end{equation}
for some constant $C\IETExtn{n}$ depending only on $\angle(E_c^n\IETExt , \orth{\Kpim{n}})$, where $\pi_{\Kpim{n}}$ denotes the projection from $\R^\A$ onto $ \Kpim{n}.$

\smallskip
The result now follows from Lemma \ref{lem:bounds_central} and \eqref{eq:bound_projection}.
\end{proof}

\subsection{Full measure of the BC and HS conditions.}\label{sec:fullmeasureproof}
We now have all the ingredients to present the proof of full measure of the BC and HS conditions, i.e.~Propostion~\ref{prop:fullmeasure}. 
\begin{proof}[Proof of Proposition~\ref{prop:fullmeasure}]
By Lemma \ref{lem: bound_cocycle_central}, there exists a measurable function 
\[\mathcal{C}: \DomExtNorm \to (0, +\infty),\]
 such that for a.e. $\IETExt \in \DomExtNorm$ and for any $n \geq 0$, 
\[ \left \|\HC{0}{n}\IETExt \mid_{E_c\IETExt } \right \| \leq \mathcal{C}\IETExtn{n}.\]
Clearly, it is enough to show that almost every IET in every fixed Rauzy class verifies the conditions, so let us fix a Rauzy class $\mathfrak{R}$ and a permutation $\pi^*$ in $\mathfrak{R}$.

 Let $\gamma$ be a finite path in the Rauzy-graph, starting and ending at $\pi^*$, such that $A_\gamma := \ZC{0}{|\gamma|}(\pi^*, \lambda)$ is a positive matrix, where {$N:=|\gamma|$} denotes the length of $\gamma$, and $\lambda$ belongs to the set $\Delta^*$ of length vectors $\lambda \in \Delta_\A$ such that $(\pi^*, \lambda)$ satisfies Keane's condition and its combinatorial rotation number starts by { $\gamma \star \gamma$, where $\star$ denotes the justapposition} of two Rauzy paths. 
For any $\pi \in \PermSpace$, define 
\[\Xi_\pi := \left\{ \tau \in \R^A \,\left|\, \, \frac{h_\alpha}{2} < \sum_{\pi_0(\beta) < \pi_0(\alpha)} \tau_\beta, \text{ for all } \alpha \in \A \text{ s.t. } \pi_0(\alpha) \neq 1, d \right. \right\}, \]
{ where  $h_\alpha:=-\frac{1}{2}(\Omega_\pi \tau)_\alpha$ is the height of the zippered rectangle labelled $\alpha\in \mathcal{A}$.  }
Then there exists a positive measure set 
\[Y \subset \big\{ \IETExt \in \DomExtNorm \mid \pi = \pi^*;\, \lambda \in \Delta^* \text{ and } \tau^{(N)} \in \Xi_{\pi^*} \}.\] Moreover, { since $\mathcal{C}$ is measurable and $Y$ has positive measure, by Luzin's theorem, by replacing $Y$ with a smaller subset,} we can assume WLOG that $\mathcal{C}$ restricted to $\ZorichMapExt^N(Y)$ is uniformly bounded by a constant $V > 0$.

By ergodicity of the extended Zorich cocycle, for a.e. $\IETExt \in \DomExtNorm$ with $\pi \in \mathfrak{R}$ there exists an increasing sequence $(m_k)_{k \in \N} \subset \N$ such that $\IETExtn{(m_k)} \in Y$, for all $k \in \N$. In particular, for any $k \in \N$, { since $|\gamma \star \gamma|=2N$,
$$
\ZC{0}{2N}\IETExtn{(m_k)} = A_\gamma\star\gamma = A_\gamma A_\gamma, 
$$
so that}
\begin{align}
\label{eq:double_matrix}
& \ZC{0}{N}\IETExtn{(m_k)} = A_\gamma = \ZC{0}{N}\IETExtn{(m_k + N)},\\
\label{eq:bounded_central_space}
& \left \|\HC{0}{m_k + N}\IETExt \mid_{E_{c}\IETExt } \right \| \leq V, \\
\label{eq:good_tau}
 & \tau^{(m_k + N)} \in \Xi_{\pi^{*}}.
 \end{align}
Since for a.e. $\IETExt \in \DomExtNorm,$
\[ \sup_{n \geq 1} \left \|\HC{0}{n}\IETExt \mid_{E_{s}\IETExt} \right \| < +\infty,\]
and by a standard Fubini argument, a full measure set in $\DomExtNorm$ gives a full measure set of IETs in $\DomNorm$ with the same forward cocycle matrices (see e.g.~\cite{ghazouani_priori_2021}), it follows from the second equality in \eqref{eq:double_matrix}, \eqref{eq:bounded_central_space} and \eqref{eq:relation_flag_splitting} that a.e. $\IETExt$ with $\pi \in \mathfrak{R}$ verifies the BC Condition along the subsequence $(n_k)_{k\in \mathbb{N}}$ given by $n_k:= m_k+N$. 

{We claim that the HS Condition also holds along the subsequence $(n_k)_{k\in \mathbb{N}}$ given by $n_k:= m_k+N$.  It is indeed also standard to check that the first equality in \eqref{eq:double_matrix}} implies assertion \eqref{cond:balanced_heights} for some $C > 0$ depending only on $A_\gamma.$ Furthermore,  assertion \eqref{cond:continuous_iterates} holds by \eqref{eq:good_tau}, see e.g.~Chapter 15 in \cite{viana_ergodic_2006}.  Thus, a.e. $\IETExt$ with $\pi \in \mathfrak{R}$ verifies the HS Condition.

\end{proof}

\section{Appendix}\label{app:Cobo}
In this section, we provide an alternative proof of Theorem \ref{thm:regularity}, by showing that it follows as an application of Theorem \ref{thm:topconjugacy} from a result by M. Cobo (which in turns exploits the work of W. Veech \cite{veech_gauss_1982}) which says the following.

\begin{theorem}[Theorem 1 in \cite{cobo_piece-wise_2002}]
\label{thm:Cobo}
For almost every IET $T_0$, for any $\omega \in E_{cs}(T_0) \,\setminus\, E_s(T_0)$ and for any AIET $T \in \textup{Aff}(\gamma(T_0), \omega)$, any conjugating map between $T_0$ and  $T$ is not an absolutely continuous function.
\end{theorem}

\begin{proof}[Alternative proof of Theorem \ref{thm:regularity} using M. Cobo's work in  \cite{cobo_piece-wise_2002}]
Notice that for a uniquely ergodic AIET $T$, its unique invariant measure $\mu$ is either singular or absolutely continuous with respect to the Lebesgue measure. Indeed, expressing $\mu = \mu_0 + \mu_1$, where $\mu_0 \ll \textup{Leb}$ and $\mu_1 \bot \textup{Leb}$, and since $T$ preserves the sets of zero Lebesgue measure, it follows that
\[ T_* \mu = T_* \mu_0 + T_*\mu_1 = \mu_0 + \mu_1, \quad T_* \mu_0 \ll \textup{Leb}, \quad T_* \mu_1 \bot \textup{Leb}.\] Hence $T_* \mu_0 = \mu_0$ and $T_* \mu_1 = \mu_1.$ By unique ergodicity, either $\mu_0$ or $\mu_1$ is zero.

 Let $T_0$ and $T$ as in Theorem \ref{thm:topconjugacy}. 
By Theorem \ref{thm:Cobo}, we may assume WLOG that the map conjugating $T_0$ and $T$ is not absolutely continuous. Moreover, since almost every IET is uniquely ergodic (see \cite{masur_interval_1982}, \cite{veech_gauss_1982}), we may assume WLOG that $T_0$ (and hence $T$) is uniquely ergodic. Since the unique invariant probability measure $\mu$ of $T$ is the push-forward of the Lebesgue measure by the conjugating map, and this map is not absolutely continuous, it follows that the $\mu$ is not absolutely continuous with respect to Lebesgue. By the argument at the beginning of the proof, it follows that $\mu$ is singular with respect to the Lebesgue measure. 
\end{proof}

\subsubsection{Acknowledgements} We are indebted to Selim Ghazouani for inspiring discussions. The authors acknowledge the support of the {\it Swiss National Science Foundation} through Grant $200021\_188617/1$. {}

\bibliographystyle{acm}
\bibliography{IET.bib, Bibliography.bib}
\end{document}